\documentclass[letterpaper,11pt,reqno]{amsart}
\usepackage[utf8]{inputenc}
\usepackage{amsmath}
\usepackage{amsthm}
\usepackage{amssymb}
\usepackage{amsfonts,mathrsfs}
\usepackage{stmaryrd}
\usepackage{amsxtra}  
\usepackage{epsfig}
\usepackage{verbatim}
\usepackage{enumerate}
\usepackage{xcolor}
\usepackage{enumitem}
\usepackage{bigints}
\usepackage{mathrsfs}
\usepackage{tikz-cd}
\usepackage{bm,bbm}
\usepackage{hyperref}
\usepackage[all]{xy}
\usepackage{mathtools, hyperref}
\usepackage{tikz}
\usepackage{tikz-cd}
\usetikzlibrary{shapes}
\usetikzlibrary{plotmarks}

\theoremstyle{plain}
\newtheorem{theorem}[equation]{Theorem}
\newtheorem{proposition}[equation]{Proposition}
\newtheorem{lemma}[equation]{Lemma}
\newtheorem{corollary}[equation]{Corollary}
\theoremstyle{definition}
\newtheorem{definition}[equation]{Definition}

\theoremstyle{remark}

\theoremstyle{remark}
\newtheorem{remark}[equation]{Remark}
\numberwithin{equation}{section}

\newcommand{\one}{\mathbbm{1}}

\newcommand{\dbar}{\bar \partial}

\newcommand{\abs}[1]{\left\vert#1\right\vert}
\newcommand{\norm}[1]{\left\Vert#1\right\Vert}
\newcommand{\ol}{\overline}

\newcommand{\wt}{\widetilde}

\newcommand{\mug}{{\mu_{\gamma}}}
\newcommand{\mugj}{{\mu_{\gamma_j}}}

\newcommand{\ipr}[1]{\left\langle #1 \right\rangle}

\newcommand{\xdownarrow}[1]{%
  {\left\downarrow\vbox to #1{}\right.\kern-\nulldelimiterspace}
}

\newcommand{\ca}{{\mathcal A}}

\newcommand{\cs}{{\mathcal S}}

\newcommand{\cx}{{\mathbb C}}

\newcommand{\rl}{\mathbb R}

\newcommand{\C}{{\mathbb C}}
\newcommand{\D}{{\mathbb D}}

\newcommand{\h}{{\mathbb H}}

\newcommand{\N}{{\mathbb N}}

\newcommand{\Q}{{\mathbb Q}}
\newcommand{\R}{{\mathbb R}}
\newcommand{\T}{{\mathbb T}}

\newcommand{\Z}{{\mathbb Z}}
\newcommand{\Uu}{\mathscr{U}}

\DeclarePairedDelimiter\floor{\lfloor}{\rfloor}

\begin{document}
\title[Projections onto Bergman spaces]{Projections onto $L^p$-Bergman spaces of Reinhardt Domains}
\author{Debraj Chakrabarti \and Luke D. Edholm}
\begin{abstract}
For $1<p<\infty$, we emulate the Bergman projection on Reinhardt domains by using a Banach-space basis of $L^p$-Bergman space. 
The construction gives  an integral kernel generalizing the  ($L^2$) Bergman kernel. 
The operator defined by the kernel is shown to be an absolutely bounded projection on the $L^p$-Bergman space  on a class of domains where the $L^p$-boundedness of the Bergman projection fails for certain $p \neq 2$.
As an application, we identify the duals of these $L^p$-Bergman spaces with weighted Bergman spaces.
\end{abstract}

\thanks{The first author was supported in part by US National Science
Foundation grant number DMS-2153907, and by a gift from the Simons 
Foundation (number 706445).}
\thanks{The second author was supported in part by Austrian Science Fund 
(FWF): AI0455721}
\address{Department of Mathematics\\Central Michigan University\\Mt. Pleasant, MI 48859, USA}
\email{chakr2d@cmich.edu}
\address{Department of Mathematics\\Universit\"at Wien, Vienna, Austria}
\email{luke.david.edholm@univie.ac.at}
\maketitle 


\section{Introduction}
\subsection{The Bergman projection on \texorpdfstring{$L^p$}{lp}}
\label{sec-introintro} Given a domain $\Omega \subset \C^n$, the {\em Bergman projection} $\bm{B}^\Omega$ is the orthogonal projection from $L^2(\Omega)$ onto the \emph{Bergman space} $A^2(\Omega)=L^2(\Omega)\cap \mathcal{O}(\Omega)$, the  subspace of square-integrable holomorphic functions. 
The Bergman projection  can be represented by integration against the \emph{Bergman kernel} $B^\Omega$: 
\begin{equation}\label{eq-bergmanproj}
\bm{B}^\Omega f(z)= \int_\Omega B^\Omega(z,w) f(w)dV(w),\qquad f\in L^2(\Omega),
\end{equation}
where $dV$ is Lebesgue measure.  The Bergman kernel enjoys 
remarkable reproducing, invariance and extremal properties and is closely 
related to the $\dbar$-Neumann problem 
(see e.g.\! \cite{bergmanbook, follandkohn, krantzbergman}). 

Bergman spaces can be naturally defined on all complex manifolds, in contrast with  Hardy spaces,  whose construction is tied to distinguished measures on the boundary of a domain, e.g., the Haar measure on the unit circle in the case of the classical Hardy space $H^p(\D)$ of $L^p$ boundary values of holomorphic functions.

Inspired by  Hardy spaces, it is natural to consider the space of $p$-th power integrable holomorphic functions $A^p(\Omega)$ of a domain $\Omega\subset \cx^n$. 
These have been known as \emph{($L^p$-)  Bergman spaces} since the 1970s, though S.~Bergman only studied the square integrable setting. 
In view of M.~Riesz's classical result  on the $L^p$-boundedness of the Szeg\H{o} projection for $1<p<\infty$, it is also natural to ask whether the Bergman projection extends to a bounded linear projection from $L^p(\Omega)$ onto $A^p(\Omega)$ via the integral formula \eqref{eq-bergmanproj}.  
When $\Omega$ is a ball in $\cx^n$, this turns out to be the case (see \cite{ZahJud64,ForRud74}); the same remains true in many classes of smoothly bounded pseudoconvex domains (\cite{PhoSte77,NagRosSteWai89,McNSte94,McNeal94} etc.)
In these cases, the extended operator turns out to be even \emph{absolutely bounded}, in the sense that the associated ``absolute" operator $(\bm{B}^\Omega)^+$ is bounded on $L^p(\Omega)$, where
\[
(\bm{B}^\Omega)^+ f(z)= \int_\Omega \abs{B^\Omega(z,w)} f(w)dV(w), \qquad f\in L^p(\Omega).
\]

On the other hand, there are examples of domains for which the extended Bergman projection fails to  define  a bounded projection from $L^p(\Omega)$ onto $A^p(\Omega)$ for some (and sometimes for all) $p\not=2$; see \cite{barrett84, BekBon95, krantz1, krantz2, zeytuncu2013, EdhMcN16b} and the survey \cite{zeytuncu2020}. 
Recent studies of the Bergman projection in certain classes of Reinhardt domains (\cite{chakzeytuncu, Edh16, EdhMcN16, Chen17, ChEdMc19, EdhMcN20, HuoWick2020a, zhang1, zhang2, monguzzi, bcem} etc.) shed more light on this phenomenon, revealing that the $L^p$-behavior of the Bergman projection that one sees on, e.g., smooth bounded strongly pseudoconvex domains breaks down on bounded Reinhardt domains whose boundary passes through the center of rotational symmetry, a simple example being the Hartogs triangle $\{\abs{z_1}<\abs{z_2}<1\}\subset \cx^2$. 
On such a domain it is possible that there are indices $1<p_1<p_2<\infty$ such that the linear subspace $A^{p_2}(\Omega)$ is not dense in the Bergman space $A^{p_1}(\Omega)$. 
This phenomenon can never occur on smoothly bounded pseudoconvex domains (see \cite{catlin}), and may constitute a glimpse of an  $L^p$-function theory where the Banach geometry of $L^p$ replaces the  Hilbert space idea of orthogonality. 
In the Reinhardt domains studied in this paper, Laurent representations are used to clarify some of these phenomena. 
For example, the fact that $A^{p_2}(\Omega)$ is not necessarily dense in $A^{p_1}(\Omega)$ is a manifestation of the fact that there may be monomials whose $p_1$-th power is integrable but not the $p_2$-th power. 

\subsection{Projection operators associated to bases}\label{sec-intro-projop} 
Let $L$ be a separable Hilbert space, $A$ a closed subspace of $L$ and $\{e_j\}$ a complete orthogonal set in $A$.
The orthogonal projection $\bm{P}$ from $L$ to $A$ may be represented by the following series (convergent in the norm of $L$):
\begin{equation}\label{eq-orthogonal}
\bm{P}f = \sum_j \frac{\ipr{f,e_j} }{\norm{ e_j}^2}e_j, \qquad f\in L.
\end{equation}
Since $\bm{P}f$ is defined geometrically as the point in $A$ nearest to $f$, this representation is independent of the choice of complete orthogonal set $\{e_j\}$.
When $L=L^2(\Omega),\, A=A^2(\Omega)$, \eqref{eq-orthogonal} coincides with the Bergman projection formula given by \eqref{eq-bergmanproj}.

In a general Banach space, the analog of a complete orthogonal set is a \emph{Schauder basis}:
a sequence $\{e_j\}_{j=1}^\infty$ in a complex Banach space $A$ is a Schauder basis if for each $f \in A$, there is a unique sequence $\{c_j\}_{j=1}^\infty$ of complex numbers such that $f = \sum_{j=1}^\infty c_je_j,$ where the series converges in the norm-topology of $A$ (see \cite{lt}). 
In this case, there exist bounded linear functionals $a_j:A\to \cx$ such that $c_j=a_j(f)$, generalizing the Fourier coefficients $a_j(f)=\frac{\ipr{f,e_j}}{\norm{e_j}^2}$ seen in the Hilbert setting.  

When $L$ is a Banach space, $A$ a closed subspace, and $\{e_j\}_{j=1}^\infty$ a Schauder basis of $A$, one might attempt to define a projection operator from $L$ onto $A$ by emulating \eqref{eq-orthogonal}:
\begin{equation}\label{eq-basisproj-def1} 
\bm{P}f = \sum_j \,\wt{a}_j(f) e_j, \qquad f\in L,
\end{equation}
where $\wt{a}_j:L\to \cx$ is a Hahn-Banach (norm-preserving) extension of $a_j:A \to \C$.
When it exists, an operator of type \eqref{eq-basisproj-def1} will be called a \emph{basis projection} determined by the Schauder basis; this notion encapsulates the orthogonal projection \eqref{eq-orthogonal} when $L$ is Hilbert. 
A less obvious example of a basis projection is seen by considering the unit circle $\T$ with the Haar measure and $1<p<\infty$. 
The classical Szeg\H{o} projection from $L^p(\T)$ onto the Hardy space $H^p(\D)$ is a basis projection; see Proposition~\ref{prop-szego}.
In contrast, we show in Proposition~\ref{cor-bergmannotbasis} that for $p \neq 2$, the attempt to extend the Bergman projection to $L^p$ by continuity -- even if successful -- is \emph{never} a basis projection. 
This is an underlying reason for the deficiencies of the Bergman projection in $L^p$ spaces, and our goal in this paper is to construct basis projections from $L^p(\Omega)$ to $A^p(\Omega)$. 

\subsection{The Monomial Basis Projection} 
Formula \eqref{eq-basisproj-def1} is purely formal, as there is no guarantee that a basis projection onto the subspace determined by a given basis exists.
Several technical points must first be addressed:

\begin{enumerate}[wide]
\item \label{Technical-1}
A basis projection depends on both the range subspace $A$ and on the choice of Schauder basis -- or the slightly more general notion of a {\em Banach-space basis} (see Section \ref{sec-schauderdef}) -- determining the projection. 
A Banach space need not have such a basis, but in the Bergman space $A^p(\Omega)$ of a Reinhardt domain  $\Omega\subset \cx^n$, there is a distinguished basis tied to geometry and function theory. 
This is the collection of Laurent monomials in $A^p(\Omega)$, functions $z\mapsto z_1^{\alpha_1}z_2^{\alpha_2} \dots z_n^{\alpha_n}$ where  $\alpha_j\in \Z$, $1\leq j \leq n$. 
The fact that these monomials under an appropriate partial ordering give a Banach-space basis of $A^p(\Omega)$ was first proved in \cite{ChEdMc19}, and is recalled in a slightly more general form in Theorem~\ref{T:monomials-form-a-Schauder-basis} below.
The projection operator from $L^p(\Omega)$ to $A^p(\Omega)$ defined in terms of this monomial basis by formula \eqref{eq-basisproj-def1} is the main topic of this paper: the \emph{Monomial Basis Projection (MBP)}. 

\item A Hahn-Banach extension of a linear functional in general is far from unique, but in our application, where we extend coefficient functionals defined on $A^p(\Omega)$ to $L^p(\Omega)$,  {we do have uniqueness}; see  Propositions~\ref{prop-taylor} and \ref{prop-lpstrictconvex} below.
This means the MBP can be unambiguously defined by \eqref{eq-basisproj-def1}, since the summation procedure is specified by the partial ordering of our Banach-space basis mentioned in item \eqref{Technical-1}.

\item None of the above guarantees that the formal series \eqref{eq-basisproj-def1} converges for $f \in L$. Showing that \eqref{eq-basisproj-def1} defines a bounded operator on $L$ requires direct estimation to show that the partial summation operators are uniformly bounded in the operator norm of $L$. 
In our application to Bergman spaces $A^p(\Omega)$, the problem is simplified because of the availability of an integral kernel representation of the MBP.
\end{enumerate}

\subsection{Notation, definitions and conventions}\label{sec-notation} 
\begin{enumerate}[wide]
\item Unless otherwise indicated, $\Omega$ will denote a bounded Reinhardt domain in $\cx^n$ with center of symmetry at 0, i.e., whenever $z\in \Omega$, for every tuple $(\theta_1,\dots, \theta_n)\in \rl^n$, we have $( e^{i\theta_1}z_1,\dots, e^{i\theta_n}z_n)\in \Omega$. Let $\abs{\Omega}\subset \rl^n$ denote its \emph{Reinhardt Shadow}, i.e., 
\begin{equation*}
\abs{\Omega}= \{ (\abs{z_1},\dots, \abs{z_n})\in \rl^n: z\in \Omega\}.
\end{equation*}
    
\item The index $p$ satisfies $1<p<\infty$, and denote by $q$ the index Hölder-conjugate to $p$, i.e., $\frac{1}{p}+\frac{1}{q}=1.$

\item \label{item-admissible} For a domain $U\subset \cx^n$ and a measurable function $\lambda:U \to [0,\infty]$ which is positive a.e. (the \emph{weight}), we set for a measurable function $f$,
\begin{equation}\label{eq-plambda}
\norm{f}_{L^p(U,\lambda)}^p= \norm{f}_{p,\lambda}^p = \int_{U} \abs{f}^p \lambda \,dV,
\end{equation}
where $dV$ denotes Lebesgue measure, and functions equal a.e. are identified.
We let $L^p(U,\lambda)$ be the space of functions $f$ for which $\norm{f}_{p,\lambda}<\infty$, which is a Banach space.

Let $A^p(U,\lambda)$ be the subspace of $L^p(U,\lambda)$ consisting of holomorphic functions:
\[
A^p(U,\lambda)=L^p(U,\lambda)\cap \mathcal{O}(U). 
\]
We will only consider weights $\lambda:U \to [0,\infty]$ which are \emph{admissible} in the sense that \emph{Bergman's inequality} holds in $A^p(U, \lambda)$, i.e., for each compact set $K\subset U$, there is a constant $C_K>0$ such that for each $f\in A^p(U, \lambda)$ we have 
\begin{equation}\label{eq-bergmanineq}
\sup_K \abs{f}\leq C_K \norm{f}_{L^p(U, \lambda)}.
\end{equation}
It is easy to see  that if $\lambda$ is a positive continuous function on $U$ then it is an admissible weight on $U$.
We treat a class of  more general admissible weights in Section \ref{SS:GenAdmWeights}. 

If $\lambda$ is an admissible weight on $U$, a standard argument shows that $A^p(U, \lambda)$ is a closed subspace of $L^p(U,\lambda)$, and therefore a Banach space. 
It is called a \emph{weighted Bergman space}.

\item We are interested in Reinhardt domains $\Omega$ and phenomena which are invariant under rotational symmetry. 
Therefore, we consider only weights $\lambda$ on $\Omega$ which are both admissible and \emph{multi-radial}, in the sense that there is a function $\ell$ on the Reinhardt shadow $\abs{\Omega}$ such that $\lambda(z_1,\dots, z_n)=\ell(\abs{z_1},\dots, \abs{z_n})$.

\item For $\alpha\in \Z^n$, we denote by $e_\alpha$ the Laurent monomial of exponent $\alpha$:
\begin{equation}\label{E:def-of-monomial-e_alpha}
e_\alpha(z)=z_1^{\alpha_1}\dots z_n^{\alpha_n}.
\end{equation}

\item We define the set of {\em $p$-allowable indices} to be the collection
\begin{equation}\label{E:def-of-p-allowable-indicies}
\mathcal{S}_p(\Omega,\lambda) = 
\left\{\alpha\in \Z^n: e_\alpha\in A^p(\Omega, \lambda)\right\}.
\end{equation}
If $\lambda \equiv 1$, we abbreviate $\cs_p(\Omega,1)$ by  $\cs_p(\Omega)$. 

\item The map $\chi_p:\cx^n\to \cx^n$ defined by
\begin{equation}\label{eq-chip}
\chi_p(\zeta) = \left(\zeta_1\abs{\zeta_1}^{p-2},\cdots,{\zeta}_n\abs{\zeta_n}^{p-2}\right)
\end{equation}
will be referred to as the \emph{twisting map}. 
It appears in the definition of the Monomial Basis Kernel in \eqref{E:def-of-monomial-basis-kernel}, and arises also in the duality pairing \eqref{eq-new-pairing}.   
Given a function $f$ we denote by $\chi_p^*f$ its pullback under $\chi_p$:
\begin{equation}\label{eq-chipstar}
    \chi_p^* f = f \circ \chi_p.
\end{equation}
\end{enumerate}

\subsection{The Monomial Basis Kernel}\label{SS:MBK}   
When it exists, the MBP of $A^p(\Omega,\lambda)$ is (by construction) a bounded surjective projection, which we write $\bm{P}^\Omega_{p,\lambda}: L^p(\Omega, \lambda)\to A^p(\Omega, \lambda)$.
To obtain an integral formula analogous to \eqref{eq-bergmanproj}, we define the \emph{Monomial Basis Kernel} of $A^p(\Omega,\lambda)$ (abbreviated \emph{MBK}), as the formal series on $\Omega \times \Omega$ given by
\begin{equation}\label{E:def-of-monomial-basis-kernel}
K_{p,\lambda}^\Omega(z,w) =\sum_{\alpha\in \mathcal{S}_p(\Omega, \lambda) }\frac{e_\alpha(z) \ol{\chi_p^*e_\alpha(w)}} {\norm{e_\alpha}_{p,\lambda}^p}. 
\end{equation}

When $p=2$, the MBK coincides with the Bergman kernel of $A^2(\Omega,\lambda)$, in which case the above series is known to converge locally normally on $\Omega \times \Omega$. 
For a general $1<p<\infty$, we show in Theorem~\ref{thm-MBK1} that when $\Omega$ is pseudoconvex, the series \eqref{E:def-of-monomial-basis-kernel} also converges locally normally on $\Omega \times \Omega$.
In Theorem~\ref{thm-mbpmbk} we prove that the MBP admits the representation
\begin{equation}\label{eq-MBPintegral}
\bm{P}_{p,\lambda}^\Omega(f)(z) = \int_\Omega  K_{p,\lambda}^\Omega(z,w) f(w) \lambda(w) dV(w), \qquad f \in L^p(\Omega,\lambda).
\end{equation}

\subsection{Improved \texorpdfstring{$L^p$}{lp}-mapping behavior}\label{SS:Intro-monomial-polyhedra}
The main theme of this paper is that the Monomial Basis Projection can have better mapping properties in $L^p$ spaces than the Bergman projection.
In Section~\ref{S:MonomialPolyhedra} we illustrate this on nonsmooth pseudoconvex Reinhardt domains called \emph{monomial polyhedra} (see \cite{NagPraDuke09,bcem}).
A bounded domain $\Uu\subset\cx^n$ is a monomial polyhedron in our sense, if there are exactly $n$ monomials $e_{\alpha^1}, \dots, e_{\alpha^n}$ such that
\[
\Uu = \left\{z\in \cx^n:\abs{e_{\alpha^1}(z)}<1, \dots, \abs{e_{\alpha^n}(z)}<1\right\}.
\]
We recall the $L^p$-mapping behavior of the Bergman projection on $\Uu$:
\begin{proposition}[\cite{bcem}]\label{P:BergmanLpRange-MonoPoly}
There is a positive integer $\kappa(\Uu)$ such that the Bergman projection  on $\Uu$ is bounded in the $L^p$-norm if and only if
\begin{equation}\label{E:BergmanLpRange-MonoPoly}
\frac{2\kappa(\Uu)}{\kappa(\Uu)+1}  < p < \frac{2\kappa(\Uu)}{\kappa(\Uu)-1}.
\end{equation}
\end{proposition}
Examples of monomial polyhedra in $\C^2$ are the (rational) generalized Hartogs triangles studied in \cite{EdhMcN16,EdhMcN16b}. 
Define $\h_\gamma=\{|z_1|^\gamma<|z_2|<1 \}$, $\gamma > 0$. 
If $\gamma = \frac{m}{n}$ is rational, $\gcd(m,n)=1$, this domain is a monomial polyhedron with $\alpha^1 = (m,-n), \alpha^2 = (0,1)$. 
In this case it can be shown that $\kappa(\h_{m/n}) = m+n$, yielding the interval $p \in \big(\frac{2m+2n}{m+n+1},\frac{2m+2n}{m+n-1}\big)$ from \eqref{E:BergmanLpRange-MonoPoly} on which the Bergman projection is $L^p$-bounded.
We also note the case of $\h_\gamma$, $\gamma$ irrational -- which is {\em not a monomial polyhedron} by our definition.
On these domains, it is shown in \cite{EdhMcN16b} that the Bergman projection is $L^p$-bounded if and only if $p = 2$.

This limited $L^p$-regularity is one of several deficiencies that can arise when the Bergman projection acts on $L^p$ spaces of nonsmooth domains; other possible defects such as a lack of surjectivity onto $A^p$ are discussed in Section \ref{S:BergmanComparison}.
The Monomial Basis Projection avoids these defects and is shown to have far more favorable mapping behavior.
Define for $1<p<\infty$ the corresponding ``absolute" operator of $A^p(\Uu)$ by
\begin{equation}\label{E:unweightedAMBO}
(\bm{P}^{\Uu}_{p,1})^+(f)(z) = \int_{\Uu}  \abs{K_{p,1}^\Uu(z,w)} f(w) \,dV(w).
\end{equation}

\begin{theorem}\label{T:MBP-AbsBoundedness}
Let $1<p<\infty$ and let $\Uu\subset \cx^n$ be a monomial polyhedron. Then the operator $(\bm{P}^\Uu_{p,1})^+$ is bounded from $L^p(\Uu)$ to itself.
\end{theorem}
After setting the stage in Sections \ref{sec-onedim} and \ref{sec-monomialtransformation}, the proof of Theorem~\ref{T:MBP-AbsBoundedness} is finally carried out in Section \ref{S:MonomialPolyhedra}.
An application of this result is given in Section \ref{S:Duality}, where we represent the dual space $A^p(\Uu)'$ as a {\em weighted} Bergman space on $\Uu$; see Theorem~\ref{T:MBP-duality-pairing}.
\begin{corollary}\label{C:MBP-boundedness/surjectivity}
The Monomial Basis Projection is a bounded surjective projection operator $\bm{P}^\Uu_{p,1}: L^p(\Uu) \to A^p(\Uu)$.
\end{corollary}
\begin{proof}
It is clear that the boundedness of the operator $(\bm{P}^\Uu_{p,1})^+$ on $L^p(\Uu)$ implies the boundedness on $L^p(\Uu)$ of
the integral operator in \eqref{eq-MBPintegral}. However, in Proposition \ref{prop:MBP-is-a-basis-projection}, we will show that whenever this integral operator satisfies $L^p$ estimates, it coincides with the  Monomial Basis Projection $\bm{P}^\Uu_{p,1}: L^p(\Uu) \to A^p(\Uu)$. The MBP is a surjective projection operator 
whenever its defining series \eqref{E:def-of-MBP-2} converges.
\end{proof}

\subsection{Acknowledgements}
The authors thank Željko Čučković, Bernhard Lamel, László Lempert, Jeff McNeal and Brett Wick for their comments and suggestions, which led to mathematical and organizational improvements in this paper.
We also thank the referee for carefully reading the paper and providing constructive suggestions.


\section{Basis Projections}\label{sec-basisproj}
 
\subsection{Bases in Banach spaces}\label{sec-schauderdef}
Since our application uses bases indexed by multi-indices, we need a slightly more general notion of a  basis in a Banach space than that of a Schauder basis described in Section~\ref{sec-intro-projop}.
For a multi-index $\alpha\in \Z^n$, let $\abs{\alpha}_\infty= \max_{1\leq j \leq n} \abs{\alpha_j}$.
 
\begin{definition}\label{def-schauderbasis}
Let $A$ be a Banach space, $n$ a positive integer and $\mathfrak{A} \subset \Z^n$ a set of multi-indices. 
A collection $\{e_\alpha: \alpha\in \mathfrak{A}\}$ of elements of $A$ is said to form a \emph{Banach-space basis} of $A$ if for each $f \in A$, there are unique complex numbers $\{c_\alpha: \alpha\in\mathfrak{A}\}$ such that
\begin{equation}\label{eq-schauderseries}  
f = \lim_{N\to \infty} \sum_{\substack{\abs{\alpha}_\infty\leq N\\\alpha\in \mathfrak{A}}} c_\alpha e_\alpha,  
\end{equation}
where the sequence of partial sums converges to $f$ in the norm-topology of $A$.
The sums on the right hand side of \eqref{eq-schauderseries} whose limit is taken are called  \emph{square partial sums}. 
\end{definition}

Schauder bases  are special cases of this definition corresponding to taking $n=1$ and
$\mathfrak{A}$ the set of positive integers. A related notion is that of a finite dimensional \emph{Schauder decomposition} (see \cite{lt}). A Banach-space basis in our sense determines a Schauder decomposition of the Banach space $A$
into the finite-dimensional subspaces $A_n = \mathrm{span}\{e_\alpha: \abs{\alpha}_\infty=n\}$, $n\geq 0$.

Adapting a classical proof (\cite[Proposition~1.a.2]{lt}), is not difficult to see that for each $\alpha\in \mathfrak{A}$, the map $a_\alpha:A\to \cx$ assigning to an element $x\in A$ the coefficient $c_\alpha$ of the series \eqref{eq-schauderseries} is a bounded linear functional on $A$.
The collection of functionals $\{a_\alpha:\alpha\in \mathfrak{A}\}$ is called the set of \emph{coefficient functionals dual to the basis $\{e_\alpha:\alpha\in \mathfrak{A}\}$}. 

\subsection{Unique Hahn-Banach extension}  
Recall that a normed linear space is  said to be {\em strictly convex}, if for distinct vectors $f,g$ of unit norm, we have $\norm{f+g} < 2$.
\begin{proposition}[\cite{taylor}]\label{prop-taylor} 
If $L$ is a Banach space such that its normed dual $L'$ is strictly convex, and $f:A\to \cx$ is a bounded linear functional on a subspace $A\subset L$, then $f$ admits a {\em unique} norm-preserving extension as a linear functional on $L$.
\end{proposition}
\begin{proof}
That at least one functional extending $f$ and having the same norm exists is the content of the Hahn-Banach theorem. 
Without loss of generality, the norm of $f$ as an element of $A'$ is 1. 
Suppose that $f$ admits two distinct extensions $f_1, f_2\in L'$ such that $\norm{f_1}_{L'}=\norm{f_2}_{L'}=1$. 
Then $g=\frac{1}{2}(f_1+f_2)$ is yet another extension of $f$ to an element of $L'$, so $\norm{g}_{L'}\geq \norm{f}_{A'}=1$. 
On the other hand, thanks  to the strict convexity of $L'$, we have $\norm{g}_{L'}< \frac{1}{2}\cdot 2 =1$. This contradiction shows that $f_1=f_2$.
\end{proof}

The examples of unique Hahn-Banach extensions in this paper arise from the following:
\begin{proposition}\label{prop-lpstrictconvex} 
Let $(X, \mathcal{F},\mu)$ be a measure space, and $1<p <\infty$. The dual of $L^p(\mu)$ is strictly convex. 
\end{proposition}
\begin{proof}
Since the dual of $L^p(\mu)$ can be isometrically identified with $L^q(\mu)$ where $q$ is the exponent conjugate to $p$, it suffices to check that $L^q(\mu)$ is strictly convex. 
Let $f,g$ be distinct elements of $L^q(\mu)$ such that $\norm{f}_{q}=\norm{g}_{q}=1$. Suppose we have $\norm{f+g}_{q} =2= \norm{f}_{q}+\norm{g}_{q}$, so that we have equality in the Minkowski triangle inequality for $L^q(\mu)$. 
It is well-known that equality occurs in the Minkowski triangle inequality only if $f= c g$ for some $c>0$. 
But since $\norm{f}_{q}=\norm{g}_{q}=1$ this gives that $c=1$, which is a contradiction since $f\not =g$. 
Therefore $\norm{f+g}_{q}<2$ showing  that $L^q(\mu)$ is strictly convex. 
\end{proof}

\subsection{Basis projections}\label{SS:basis-projections}
Let $L$ be a Banach space such that its dual is strictly convex, $A$ be a closed subspace, the collection $\{e_\alpha: \alpha\in \mathfrak{A}\}$ a Banach-space basis of $A$ in the sense of Definition~\ref{def-schauderbasis}, and let $\{a_\alpha: \alpha\in \mathfrak{A}\}$ be the coefficient functionals dual to this basis. 
Let $\wt{a}_\alpha: L \to \cx$ be the unique Hahn-Banach extension of the functional $a_\alpha:A\to \cx$, where uniqueness follows by Propositon~\ref{prop-taylor}.

\begin{definition}\label{D:basis-projection}
A bounded linear projection operator $\bm{P}$ from $L$ onto $A$ is called the \emph{basis projection} determined by $\{e_\alpha: \alpha\in \mathfrak{A}\}$, if for each $f \in L$, we have a series representation convergent in the norm of $L$ given by
\begin{equation}\label{eq-domain}
\bm{P}f = \lim_{N\to \infty} \sum_{\substack{\abs{\alpha}_\infty\leq N\\ \alpha\in \mathfrak{A} }}\wt{a}_\alpha(f) e_\alpha.
\end{equation}
\end{definition}


\subsection{The Szeg\H{o} projection.}

Let $1<p<\infty$, $L=L^p(\T)$, the $L^p$-space of the circle with the normalized Haar measure $\frac{1}{2\pi} d\theta$, and  $A=H^p(\D)$, the Hardy space of the unit disc, the subspace of $L^p(\T)$ consisting of those elements of $L^p(\T)$ which are boundary values of holomorphic functions in the disc. 
Let $\tau_\alpha(e^{i\theta})= e^{i\alpha\theta},\, \alpha\in \Z,$ denote the $\alpha$-th trigonometric monomial on $\T$.
It is well-known that $\{\tau_\alpha: \alpha\geq 0\}$ is a (normalized) Schauder basis of $H^p(\D)$, i.e., the partial sums of the Fourier series of a function in  $H^p(\D)$ converge in the norm $L^p(\T)$. Notice that Schauder bases are simply Banach-space bases in the sense of Definition~\ref{def-schauderbasis} where  $\mathfrak{A}$ is the set of positive integers.
\begin{proposition} \label{prop-szego}
For $1<p<\infty$, the basis projection from $L^p(\T)$ onto $H^p(\D)$ determined by the Schauder basis $\{\tau_\alpha\}_{\alpha=0}^\infty$ exists, and coincides with the Szeg\H{o} projection.
\end{proposition}
\begin{proof}
The coefficient functionals on $H^p(\D)$ dual to the Schauder basis $\{\tau_\alpha: \alpha\geq 0\}$ are precisely the Fourier coefficient functionals $\{a_\alpha\}_{\alpha=0}^\infty$:
\begin{equation}\label{eq-fourier}
a_\alpha(f)= \int_{0}^{2\pi} f(e^{i\theta})e^{-i\alpha \theta}\frac{d\theta}{2\pi}, \qquad f\in H^p(\D).
\end{equation}
Notice that for $f\in H^p(\D)$, we have
\begin{equation}
    \label{eq-fourier2}
    \abs{a_\alpha(f)} \leq \int_{0}^{2\pi} \abs{f(e^{i\theta})}\frac{d\theta}{2\pi} \leq \norm{f}_{L^p(\T)}\norm{1}_{L^q(\T)} = \norm{f}_{L^p(\T)}, 
\end{equation}
where $q$ is the Hölder conjugate of $p$, and we use Hölder's inequality along with the fact that the measure is a probability measure. 
Therefore $\norm{a_\alpha} \leq 1$. But since $\norm{\tau_\alpha}_{L^p(\T)}=1$, and $a_\alpha(\tau_\alpha)=1$, it follows that $\norm{a_\alpha}=1$. 
We now claim that the Hahn-Banach extension $\wt{a}_\alpha:L^p(\T)\to \cx$ of the coefficient functional $a_\alpha:H^p(\D)\to \cx$ is still the Fourier coefficient functional:
\[ 
\wt{a}_\alpha(f)= \int_{0}^{2\pi} f(e^{i\theta})e^{-i\alpha \theta}\frac{d\theta}{2\pi}, \qquad f\in L^p(\T).
\]
Indeed, $\wt{a}_\alpha$ is an extension of $a_\alpha$, and repeating the argument of \eqref{eq-fourier2} shows $\norm{\wt{a}_\alpha}=1$, and thus it is a Hahn-Banach extension. 
Uniqueness follows from Propositions~\ref{prop-taylor} and \ref{prop-lpstrictconvex}. 

Let $\bm{S}$ denote the basis projection from $L^p(\T)$ onto $H^p(\D)$ and let $f\in L^p(\T)$ be a trigonometric polynomial. 
Then formula \eqref{eq-domain} in this case becomes:
\[
\bm{S}f(e^{i\phi}) = \sum_{\alpha=0}^\infty \left(\int_0^{2\pi} f(e^{i\theta})e^{-i\alpha \theta}\frac{d\theta}{2\pi}\right) e^{i\alpha\phi} = \int_0^{2\pi}\frac{f(e^{i\theta})}{1- e^{i(\phi-\theta)}}\cdot\frac{d\theta}{2\pi}.
\]
This shows that on the trigonometric polynomials, the basis projection coincides with the Szeg\H{o} projection, which is known to be represented by the singular integral at the end of the above chain of equalities. 
But as the Szeg\H{o} projection is bounded from $L^p(\T)$ onto $H^p(\D)$, it follows that the basis projection exists and equals the Szeg\H{o} projection on $L^p(\T)$.
\end{proof}

\subsection{The Monomial Basis Projection}
On a Reinhardt domain $\Omega\subset \cx^n$ each holomorphic function $f\in\mathcal{O}(\Omega)$ has a unique \emph{Laurent expansion}
\begin{equation}\label{eq-laurent}
    f=\sum_{\alpha\in \Z^n} c_\alpha e_\alpha,
\end{equation}
where  $c_\alpha\in\cx$ and the series converges locally normally, i.e., for each compact $K\subset \Omega$, the sum $\sum_\alpha \norm{c_\alpha e_\alpha}_K<\infty$, where $\norm{\cdot}_K = \sup_K\abs{\cdot}$ is the sup norm (see e.g.\!\! \cite{range}).
It follows that \eqref{eq-laurent} converges 
uniformly on compact subsets of $\Omega$. 
Define 
\begin{equation}\label{eq-laurentcoefficient}
    a_\alpha:\mathcal{O}(\Omega)\to \cx, \qquad a_\alpha(f)= c_\alpha
\end{equation}
where $c_\alpha$ is as above in \eqref{eq-laurent}. The functional $a_\alpha$ is called the $\alpha$-th \emph{Laurent coefficient functional} of the domain $\Omega$.

The following result shows that the Laurent monomials (under an appropriate ordering) form a  basis of the Bergman space $A^p(\Omega,\lambda)$, where $\lambda$ is an admissible multi-radial weight. 
The unweighted version of this result (the case $\lambda\equiv 1$) was proved in \cite{ChEdMc19}, inspired by the case of the disc considered in \cite{zhu1991}. 
The more general Theorem \ref{T:monomials-form-a-Schauder-basis} is proved in exactly the same way, by replacing the implicit weight $\lambda \equiv 1$ in \cite[Theorem~3.11]{ChEdMc19} with a general multi-radial weight $\lambda$. 
A key ingredient of the proof, the density of Laurent polynomials in $A^p(\Omega, \lambda)$, can also be proved
using Cesàro summability of power series (see \cite[Theorem~2.5]{chakdawn}.)
Recall that the notation and  conventions established in Section~\ref{sec-notation} are in force throughout the paper.
 
\begin{theorem}\label{T:monomials-form-a-Schauder-basis}
The collection of Laurent monomials $\{e_\alpha: \alpha \in \cs_p(\Omega,\lambda) \}$ forms a Banach-space basis of $A^p(\Omega,\lambda)$. 
The functionals dual to this basis are the coefficient functionals $\{a_\alpha: \alpha \in \cs_p(\Omega,\lambda) \}$, and the norm of $a_\alpha:A^p(\Omega, \lambda) \to \C$ is given by
\begin{equation}\label{eq-aalphanorm} 
\norm{{a_\alpha}}_{A^p(\Omega, \lambda)'} = \frac{1}{\norm{e_\alpha}_{p,\lambda}}.
\end{equation}
\end{theorem}

Thus, if $f\in A^p(\Omega,\lambda)$, the Laurent series of $f$ written as $\sum_{\alpha\in \Z^n} a_\alpha(f) e_\alpha$ consists only of terms corresponding to monomials $e_\alpha \in A^p(\Omega,\lambda)$, i.e., if $\alpha \not \in \mathcal{S}_p(\Omega,\lambda) $, then $a_\alpha(f) = 0$.



We are ready to formally define the main object of this paper:
\begin{definition}\label{D:def-of-MBP-1}
A bounded linear projection $\bm{P}_{p,\lambda}^\Omega$ from $L^p(\Omega,\lambda)$ onto $A^p(\Omega,\lambda)$ is called the \emph{Monomial Basis Projection} of $A^p(\Omega,\lambda)$, if for $f \in L^p(\Omega,\lambda)$ it admits the series representation convergent in the norm of $L^p(\Omega,\lambda)$ given by
\begin{equation}\label{E:def-of-MBP-2}
\bm{P}^\Omega_{p,\lambda}(f) = \lim_{N\to \infty} \sum_{\substack{\abs{\alpha}_\infty\leq N \\ \alpha\in \cs_p(\Omega,\lambda) }}\wt{a}_\alpha(f) e_\alpha,
\end{equation}
where $\wt{a}_\alpha:L^p(\Omega,\lambda) \to \C$ is the unique Hahn-Banach extension of the coefficient functional $a_\alpha:A^p(\Omega,\lambda) \to \C$.
\end{definition}

\begin{remark}\label{R:surjectivity}
The surjectivity onto the space $A^p(\Omega,\lambda)$ is built in to the definition of the Monomial Basis Projection, since it acts as the identity operator there.
Notice that the MBP is a basis projection in the sense of Definition~\ref{D:basis-projection}, when $L=L^p(\Omega), A=A^p(\Omega)$ and $\{e_\alpha\}$ is the  monomial basis of $A^p(\Omega,\lambda)$.
\hfill $\lozenge$
\end{remark}

\section{The monomial basis kernel}\label{sec-MBK}

\subsection{Existence of the kernel function} 
The Monomial Basis Kernel of $A^p(\Omega,\lambda)$ was introduced as a formal series in \eqref{E:def-of-monomial-basis-kernel}.
Using \eqref{eq-chip} and \eqref{eq-chipstar}, we can write
\begin{align}\label{eq-chistaralpha}
\chi_p^*e_\alpha(w) 
=e_\alpha(w)\abs{e_\alpha(w)}^{p-2},
\end{align}
which allows for the re-expression of the MBK  as
\begin{equation}\label{E:MBK-alt-form}
    K_{p,\lambda}^\Omega(z,w) =\sum_{\alpha\in \mathcal{S}_p(\Omega, \lambda) }\frac{e_\alpha(z) \ol{e_\alpha(w)}\abs{e_\alpha(w)}^{p-2}} {\norm{e_\alpha}_{p,\lambda}^p}.
\end{equation}
A sufficient condition for the convergence of this series is now given. 
\begin{theorem}\label{thm-MBK1}
Let $\Omega$ be a \emph{pseudoconvex}  Reinhardt domain in $\C^n$ and $\lambda$ be an admissible  multi-radial weight function on $\Omega$. The series \eqref{E:MBK-alt-form} defining $K^\Omega_{p,\lambda}(z,w)$ converges locally normally on  $\Omega \times \Omega$.  
\end{theorem}
We need two lemmas for the proof of this result. 
The first is an analog for Laurent series of Abel's lemma on the domain of convergence of a Taylor series (\cite[p. 14]{range}):

\begin{lemma}\label{lem-abel}
Let $\Omega\subset \cx^n$ be a Reinhardt domain, define $\mathcal{S}(\Omega) = \{\alpha\in \Z^n : e_\alpha \in \mathcal{O}(\Omega)\}$, and for coefficients $a_\alpha \in \cx,\,\alpha\in\mathcal{S}(\Omega)$, let 
\begin{equation}\label{eq-series}
\sum_{\alpha\in \mathcal{S}(\Omega)}a_\alpha e_\alpha
\end{equation}
be a formal Laurent series on $\Omega$. 
Suppose that for each $z\in \Omega$ there is a $C>0$ such that for each $\alpha\in \mathcal{S}(\Omega)$ we have $\abs{a_\alpha e_\alpha(z)}\leq C.$
Then \eqref{eq-series} converges locally normally on $\Omega$.
\end{lemma}
\begin{proof}
See Lemma~1.6.3 and Proposition~1.6.5 of \cite[Section~1.6]{JarPflBook08}.
\end{proof}

Given a Reinhardt domain $\Omega\subset \cx^n$ and a number $m>0$, define the {\em $m$-th Reinhardt power} of $\Omega$ to be the Reinhardt domain
\begin{equation}\label{E:ReinhardtPower}
\Omega^{(m)} = \left\{ z\in \C^n : \left(\abs{z_1}^{\frac{1}{m}}, \dots, \abs{z_n}^{\frac{1}{m}}\right) \in \Omega\right\}.
\end{equation}
If $\Omega$ is pseudoconvex, then for each $m>0$ the domain $\Omega^{(m)}$ is  pseudoconvex. 
Indeed, recall the logarithmic shadow of $\Omega$, the subset $\log(\Omega)$ of $\rl^n$ given by 
\begin{equation}\label{eq-logshadow}
\log(\Omega )= \{\left(\log\abs{z_1},\dots, \log \abs{z_n}\right): z\in \Omega\}.
\end{equation}
Recall also that  $\Omega$ is pseudoconvex if and only if the set $\log(\Omega)$ is convex, and $\Omega$ is ``weakly relatively complete" (\cite[Theorem 1.11.13 and Proposition~1.11.6]{JarPflBook08}). 
It is easily seen that the condition of weak relative completeness is preserved by the construction of Reinhardt powers, and
\[ 
\log\left(\Omega^{(m)}\right) = \{(m\log\abs{z_1},\dots, m\log\abs{z_n}): z\in \Omega\} = m\log(\Omega)
\]
is itself convex, if $\log(\Omega)$ is convex. 
So $\Omega^{(m)}$ is pseudoconvex if and only if $\Omega$ is pseudoconvex.

The second result needed in the proof of Theorem~\ref{thm-MBK1} is the following:
\begin{lemma}\label{lem-series}
Let $A$ be a Banach space of holomorphic functions on $\Omega$ and suppose that for each $z\in \Omega$ the evaluation functional $\phi_z : A \to \cx$ given by $\phi_z(f)=f(z)$ for $f\in A$ is continuous. 
Then for  $m>0$, the following series converges locally normally on $\Omega^{(m)}$:
\begin{equation*}
\sum_{\substack{\alpha\in \Z^n\\ e_\alpha\in A} } \frac{e_\alpha}{\norm{e_\alpha}_A^m}.
\end{equation*}
\end{lemma}
\begin{proof}
Let $z\in \Omega^{(m)}$ so that there is $\zeta\in \Omega$ such that $\abs{z_j}=\abs{\zeta_j}^m$ for each $j$. 
If $\phi_\zeta: A \to \cx$ is the evaluation functional, there is a constant $C>0$ such that $\abs{\phi_\zeta(f)}\leq C \norm{f}_A$ for each $f\in A$. 
Then for each $\alpha\in \Z^n$ such that $e_\alpha\in A$ we have
\begin{equation*}
\frac{\abs{e_\alpha(z)}}{\norm{e_\alpha}_A^m}= \left(\frac{\abs{e_\alpha(\zeta)}}{\norm{e_\alpha}_A}\right)^m= \left(\frac{\phi_\zeta(e_\alpha)}{\norm{e_\alpha}_A}\right)^m \leq C^m.
\end{equation*}
The result now follows by Lemma~\ref{lem-abel}.
\end{proof}

\begin{proof}[Proof of Theorem~\ref{thm-MBK1}]
Let $t_j=z_j\ol{w_j}\abs{w_j}^{p-2}$, $1\le j\le n$, and $t=(t_1,\dots,t_n)$. Then the series for the MBK given in \eqref{E:MBK-alt-form} assumes the form
\begin{equation}\label{E:tseries-MBK}
K^\Omega_{p,\lambda}(z,w)=\sum_{\alpha \in \mathcal{S}_p(\Omega,\lambda)} \frac{t^\alpha}{\norm{e_\alpha}_{p,\lambda}^p}.
\end{equation}
Since Bergman's inequality \eqref{eq-bergmanineq} holds for admissible weights by definition, point evaluations are bounded on $A^p(\Omega,\lambda)$. 
Lemma~\ref{lem-series} therefore guarantees the series in $\eqref{E:tseries-MBK}$ above converges locally  normally on $\Omega^{(p)}$  defined in \eqref{E:ReinhardtPower}. 
It thus suffices to show that the image of the map $\Omega \times \Omega \to \cx^n$ given by
\begin{equation*}
(z,w) \longmapsto (t_1,\dots, t_n)
\end{equation*}
coincides with $\Omega^{(p)}$, since then the image of a compact set $K\subset \Omega\times \Omega$ is a compact subset of $\Omega^{(p)}$, on which the series \eqref{E:tseries-MBK} is known to converge normally. 

Now consider the logarithmic shadow $\log(\Omega \times \Omega) = \log(\Omega)\times\log(\Omega)$ defined in \eqref{eq-logshadow}.
Due to the log-convexity of pseudoconvex Reinhardt domains, what we want to prove is equivalent to saying that the map from $\log(\Omega)\times \log(\Omega)\to \rl^n$ given by
\begin{equation}\label{eq-map}
(\xi, \eta) \longmapsto \xi + (p-1)\eta
\end{equation}
has image exactly $p\log(\Omega) = \{p\theta: \theta \in \log(\Omega)\} = \log\left(\Omega^{(p)}\right)$.  But since $\log(\Omega)$ is convex, the map on $\log(\Omega) \times \log(\Omega)$ given by
\begin{equation*}
(\xi, \eta) \longmapsto \tfrac{1}{p}\xi + \big(1-\tfrac{1}{p}\big)\eta 
\end{equation*}
has image contained in $\log(\Omega)$. 
Taking $\xi=\eta$ we see that the image is exactly $\log(\Omega)$. 
Therefore the image of \eqref{eq-map} is precisely $p\log(\Omega)$ and  we have proved that the series \eqref{E:MBK-alt-form} converges locally normally on $\Omega\times\Omega$.
\end{proof}

\subsection{More general admissible weights}\label{SS:GenAdmWeights}
  
Continuous positive functions $\lambda$ are always admissible weights in the sense of Section~\ref{sec-notation}, item~\eqref{item-admissible}. 
In Sections \ref{sec-onedim}, \ref{sec-monomialtransformation}, \ref{S:MonomialPolyhedra} and \ref{S:Duality} below, we encounter more general multi-radial weights which vanish or blow up along the axes.
Let $Z\subset\cx^n$ denote the union of the coordinate hyperplanes
\[ 
Z=\{z\in \cx^n:  z_j=0 \text{ for some  } 1\leq j \leq n\}.
\]
\begin{proposition}\label{prop-admissible}
    Let $U$ be a  domain in $\cx^n$ and let $U^*=U\setminus Z$. Suppose that $\lambda:U\to [0,\infty]$
    is a measurable function on $U$ such that the restriction $\lambda|_{U^*}$ is an admissible weight on $U^*$. Then 
    $\lambda$ is an admissible weight on $U$.
\end{proposition}
\begin{proof} Assume that  $U\cap Z\not=\varnothing$, since otherwise there is nothing to show,  and set $\lambda^*=\lambda|_{U^*}$. If $f\in A^p(U, \lambda)$, then since $\lambda^*$ is admissible on
$U^*$, if a compact $K$ is contained in $U^*$, there exists a $C_K>0$ such that 
\[
\sup_K \abs{f}\leq C_K \norm{f}_{A^p(U^*, \lambda^*)}=C_K \norm{f}_{A^p(U, \lambda)}.
\]
To complete the proof, we need to show that for each $\zeta\in U\cap Z$, there is a compact neighborhood $K$ of 
$\zeta$ in $U$ such that \eqref{eq-bergmanineq} holds for each $f\in A^p(U, \lambda)$. 
Now, there is a polydisc $P$ centered at $\zeta$ given by
$P=\{z\in \cx^n: \abs{z_j-\zeta_j}< r,\, 1 \leq j \leq n\}$
such that the closure $\ol{P}$ is contained in $U$. We can assume further  that the radius $r>0$ is chosen so that it is distinct from each of the nonnegative numbers $\abs{\zeta_j},\, 1\leq j \leq n$. Then the ``distinguished boundary"
\[
T=\{z\in \cx^n: \abs{z_j-\zeta_j} =r,\, 1\leq j \leq n \}
\]
of this polydisc satisfies the condition that $T\subset U^*$. 
Therefore for each $f\in \mathcal{O}(U)$ and each $w\in P$, we have the Cauchy representation:
\begin{equation}\label{eq-cauchy}
    f(w)= \frac{1}{(2\pi i)^n}\int_T \frac{f(z_1,\dots,z_n)}{(z_1-w_1)\dots (z_n-w_n)} \, dz_1\dots dz_n
\end{equation}    
where the integral is an $n$-times repeated contour integral on $T$. 
Now suppose that $K$ is a compact subset of $P$ containing the center $\zeta$, and let $\rho>0$ be such that $\abs{z_j-w_j}\geq \rho$ for each $z\in T$ and  
$w\in K$. 
Then for $w\in K$, a sup-norm estimate on \eqref{eq-cauchy} gives
\[
\abs{f(w)} \leq \frac{1}{(2\pi)^n}\cdot \frac{\sup_T\abs{f}}{\rho^n}(2\pi r)^n\leq \left(\frac{r}{\rho} \right)^n \cdot \norm{f}_{A^p(U^*,\lambda^*)} = \left(\frac{r}{\rho} \right)^n \cdot \norm{f}_{A^p(U,\lambda)}
\]
where we used the fact that $\lambda^*$ is admissible on $U^*$. The result follows.
\end{proof}


\subsection{Integral representation of the Monomial Basis Projection}\label{S:Integral-rep-MBP}
\begin{theorem}\label{thm-mbpmbk}
If the Monomial Basis Projection $\bm{P}_{p,\lambda}^\Omega:L^p(\Omega, \lambda)\to A^p(\Omega, \lambda)$ exists, then
\begin{equation}\label{E:def-of-MBP}
\bm{P}_{p,\lambda}^\Omega(f)(z) = \int_\Omega K_{p,\lambda}^\Omega(z,w)f(w)  \lambda(w) dV(w), \qquad f \in L^p(\Omega,\lambda),
\end{equation}
and for each $z\in\Omega$, we have $K_{p,\lambda}^\Omega(z, \cdot)\in L^q(\Omega, \lambda)$.
\end{theorem}

When $p = 2$, this is simply the representation of the Bergman projection $\bm{B}_\lambda^\Omega$ of $A^2(\Omega,\lambda)$ by its Bergman kernel.
But the existence of the MBP of $A^p(\Omega,\lambda)$ for $p \neq 2$ is not guaranteed by abstract Hilbert-space theory. 
We note a related consequence of Theorem \ref{thm-mbpmbk}, which should be contrasted with Proposition~\ref{prop-szego}:

\begin{corollary}\label{cor-bergmannotbasis}
Suppose the Bergman projection $\bm{B}_{\lambda}^\Omega: L^2(\Omega,\lambda)\to A^2(\Omega, \lambda)$ extends by continuity to a bounded operator $\bm{B}_{\lambda}^\Omega: L^p(\Omega, \lambda)\to A^p(\Omega, \lambda)$, $p \neq 2$.
The extension is not the basis projection determined by the monomial  basis $\{e_\alpha:\alpha \in \cs_p(\Omega,\lambda)\}$.
\end{corollary}
\begin{proof}
This is immediate, since the Bergman kernel is distinct from the MBK for $p \neq 2$.
\end{proof}
By Proposition~\ref{prop-lpstrictconvex}, the dual space of $L^p(\Omega, \lambda)$ is strictly convex. 
Proposition \ref{prop-taylor} thus guarantees that each coefficient functional in the set $\{a_\alpha: \alpha \in \mathcal{S}_p(\Omega,\lambda)\}$ dual to the monomial  basis $\{e_\alpha: \alpha \in \mathcal{S}_p(\Omega,\lambda)\}$ has a unique Hahn-Banach extension to a functional $\wt{a}_\alpha:L^p(\Omega, \lambda)\to \C$. 
We now identify this extension:

\begin{proposition}\label{prop-hbext}
For $\alpha \in \mathcal{S}_p(\Omega,\lambda)$, let $g_\alpha$ be the function defined on $\Omega$ by
\begin{equation}\label{eq-galpha}
g_\alpha =\frac{\chi_p^*e_\alpha}{\norm{e_\alpha}_{p,\lambda}^p} = \frac{e_\alpha\abs{e_\alpha}^{p-2}}{\norm{e_\alpha}_{p,\lambda}^p}.
\end{equation}
Then the unique Hahn-Banach extension $\wt{a}_\alpha: L^p(\Omega, \lambda)\to \cx$ of the coefficient functional $a_\alpha:A^p(\Omega, \lambda)\to \cx$ is given by
\begin{equation}\label{eq-aalphatilde}
\wt{a}_\alpha(f)= \int_\Omega f\cdot \ol{g_\alpha}\,\lambda \,dV, \qquad f\in L^p(\Omega, \lambda).
\end{equation}
\end{proposition}
\begin{proof} 
First we compute the norm of $g_\alpha$ in $L^q(\Omega,\lambda)$: 

\begin{align*}
\norm{g_\alpha}_{q,\lambda}^q =\frac{1}{\norm{e_\alpha}_{p,\lambda}^{pq}}\int_\Omega \abs{e_\alpha}^{(p-1)q} \lambda\, dV
=\frac{1}{\norm{e_\alpha}_{p,\lambda}^{pq}}\norm{e_\alpha}_{p,\lambda}^p= \frac{1}{\norm{e_\alpha}_{p,\lambda}^{pq-p}}=\frac{1}{\norm{e_\alpha}_{p,\lambda}^{q}}.
\end{align*}
It follows that $g_\alpha\in L^q(\Omega, \lambda) $ and the linear functional in \eqref{eq-aalphatilde} satisfies $\wt{a}_\alpha\in L^p(\Omega, \lambda)'$ with norm given by
\begin{equation}\label{eq-normaalphatilde}
\norm{\wt{a}_\alpha}_{L^p(\Omega, \lambda)'} =\norm{g_\alpha}_{q,\lambda} = \frac{1}{\norm{e_\alpha}_{p,\lambda}}.
\end{equation}
By \eqref{eq-aalphanorm}, we have $\norm{{a_\alpha}}_{A^p(\Omega, \lambda)'} = \norm{\wt{a}_\alpha}_{L^p(\Omega, \lambda)'}$.
To complete the proof it remains to show that $\wt{a}_\alpha$ is an extension of $a_\alpha$.

By Theorem \ref{T:monomials-form-a-Schauder-basis}, the linear span of $\{e_\beta: \beta\in \mathcal{S}_p(\Omega, \lambda)\}$ is dense in  $A^p(\Omega, \lambda)$. 
Therefore we only need to show that for each $\beta\in \mathcal{S}_p(\Omega, \lambda)$, we have $\wt{a}_\alpha(e_\beta)={a_\alpha}(e_\beta)$. 
Since $\lambda$ is multi-radial, there is a function $\ell$ on the Reinhardt shadow $\abs{\Omega}$ such that  $\lambda(z)=\ell(\abs{z_1},\dots, \abs{z_n})$.
And since $g_\alpha\in L^q(\Omega, \lambda)$ and $e_\beta\in L^p(\Omega, \lambda)$, the product $e_\beta\ol{g_\alpha}\in L^1(\Omega, \lambda)$.
Fubini's theorem therefore implies
\begin{equation}\label{eq-isextn}
\int_\Omega e_\beta \ol{g_\alpha}\lambda dV= \frac{1}{\norm{e_\alpha}_{p,\lambda}^p} \int_{\abs{\Omega}} r^\beta (r^\alpha)^{p-1}\left(\int_{\mathbb{T}^n }e^{i\ipr{\beta-\alpha, \theta }}d\theta\right)r_1 r_2\dots r_n \ell dr_1\dots dr_n,
\end{equation}
where $d\theta =d\theta_1\dots d\theta_n$ is the natural volume element of the unit torus $\mathbb{T}^n$.
First suppose that $\beta\not=\alpha$, so that the integral over $\mathbb{T}^n$ on the right hand side of \eqref{eq-isextn} vanishes. 
Then we have $\int_\Omega e_\beta \ol{g_\alpha}\lambda dV =0 ={a_\alpha}(e_\beta)$.
If $\beta=\alpha$, \eqref{eq-isextn} gives
\begin{align*}
\int_\Omega e_\alpha \ol{g_\alpha}\lambda dV
= \frac{(2\pi)^n}{\norm{e_\alpha}_{p,\lambda}^p}\cdot \int_{\abs{\Omega}}  (r^\alpha)^{p}r_1 r_2\dots r_n\ell  dr_1\dots dr_n
&= \frac{1}{\norm{e_\alpha}_{p,\lambda}^p} \cdot \norm{e_\alpha}_{p,\lambda}^p= 1 = a_\alpha(e_\alpha).
\end{align*}
It follows that $\wt{a}_\alpha$ is a norm preserving extension of $a_\alpha$.  Since this extension is unique, the result follows.
\end{proof}

Observe that by combining \eqref{E:MBK-alt-form} and \eqref{eq-galpha}, the MBK of $A^p(\Omega, \lambda)$ can be written as 
\begin{equation}\label{eq-kgalpha}
K_{p,\lambda}^\Omega(z,w)=\sum_{\alpha\in \cs_p(\Omega, \lambda)}e_\alpha(z) \ol{g_\alpha(w)}. 
\end{equation}

We now establish our necessary and sufficient condition for the existence of the MBP:
\begin{proposition}\label{prop:MBP-is-a-basis-projection}
Define an integral operator on $C_c(\Omega)$ by
\begin{equation}\label{eq-qdef}
\bm{Q}f(z) = \int_\Omega  K_{p,\lambda}^\Omega(z,w) f(w)\lambda(w) dV(w), \qquad f\in C_c(\Omega).
\end{equation}
The MBP of $A^p(\Omega,\lambda)$ exists if and only if $\bm{Q}$ satisfies a weighted $L^p$-estimate, i.e., there is a constant $C>0$ such that for each $f\in C_c(\Omega)$ we have the inequality
\begin{equation}\label{eq-lpest}
\norm{\bm{Q} f}_{p,\lambda}\leq C \norm{f}_{p,\lambda}.
\end{equation}
\end{proposition}
\begin{proof} Recall that  $\Omega \subset \C^n$ is a pseudoconvex Reinhardt domain and $\lambda$ is an admissible multi-radial weight. 
The function $K_{p,\lambda}^\Omega$ is continuous on $\Omega\times \Omega$ by Theorem~\ref{thm-MBK1}, 
so the integral in \eqref{eq-qdef} exists for each $z\in \Omega$. 
Since the function $z \mapsto K_{p,\lambda}^\Omega(z, w)$ is holomorphic for each $w \in \Omega$, $\bm{Q}f$ is holomorphic for $f\in C_c(\Omega)$, for instance, by applying Morera's theorem in each variable, or equivalently, by applying $\dbar$ to both sides.

Let $f\in C_c(\Omega)$. 
Since the series for $K_{p,\lambda}^\Omega$ converges absolutely and uniformly on the compact subset $\{z\}\times \mathrm{supp}(f) \subset \Omega\times \Omega$, equation \eqref{eq-kgalpha} gives
\begin{align}
\bm{Q}f(z) &= \int_\Omega \bigg(\sum_{\alpha \in \cs_p(\Omega,\lambda)} e_\alpha(z)\ol{g_\alpha(w)}\bigg)f(w)\lambda(w)\,dV(w) \notag \\
&=\sum_{\alpha \in \cs_p(\Omega,\lambda)} \left( \int_\Omega f(w)  \ol{g_\alpha(w)}\lambda(w)\,dV(w)\right)e_\alpha(z)
=\sum_{\alpha \in \cs_p(\Omega,\lambda)} \wt{a}_\alpha(f) e_\alpha(z). \label{eq-qflaurent}
\end{align}
The series \eqref{eq-qflaurent} converges unconditionally and is the Laurent series of the holomorphic function $\bm{Q}f$. It is therefore uniformly convergent for $z$ in compact subsets of $\Omega$.

Suppose now that the MBP $\bm{P}_{p,\lambda}^\Omega:L^p(\Omega, \lambda)\to A^p(\Omega, \lambda) $ exists, which by Definition~\ref{D:def-of-MBP-1} is a bounded, surjective, linear projection given by the following limit of partial sums, convergent in $A^p(\Omega,\lambda)$:
\begin{equation}\label{eq-sqpartialsums} 
\bm{P}_{p,\lambda}^\Omega f = \lim_{N\to \infty} \sum_{\substack{\abs{\alpha}_\infty\leq N\\\alpha\in \cs_p(\Omega,\lambda) }}\wt{a}_\alpha(f) e_\alpha, \qquad f \in L^p(\Omega,\lambda).
\end{equation}
Since convergence in $A^p(\Omega,\lambda)$ implies uniform convergence on compact subsets, it follows that for $f \in C_c(\Omega)$, $\bm{Q}f = \bm{P}_{p,\lambda}^\Omega f$. 
Therefore $\bm{Q}$ satisfies $L^p$-estimates, i.e. \eqref{eq-lpest} holds.

Conversely, suppose that \eqref{eq-lpest} holds.  
Then $\bm{Q}$ can then be extended by continuity to an operator $\wt{\bm{Q}}$ on $L^p(\Omega, \lambda)$ with the same norm. 
We claim that $\wt{\bm{Q}}$ is the MBP.

If $f\in L^p(\Omega, \lambda)$, we can find a sequence $\{f_j\}\subset C_c(\Omega)$ such that $f_j\to f$ in $L^p(\Omega,\lambda)$. 
Each $\bm{Q}f_j \in A^p(\Omega, \lambda)$ and (by definition) $\bm{Q}f_j\to \wt{\bm{Q}}f$ in  $L^p(\Omega, \lambda)$. 
But this implies $\bm{Q}f_j\to \wt{\bm{Q}}f$ uniformly on compact subsets, so the limit  $\wt{\bm{Q}}f$ is holomorphic, and thus the range of $\wt{\bm{Q}}$ is contained in $A^p(\Omega, \lambda)$. 
A direct computation now shows $\wt{\bm{Q}}e_\alpha = e_\alpha$ for $\alpha \in \cs_p(\Omega,\lambda)$, and it follows that $\wt{\bm{Q}}$ is a surjective projection from $L^p(\Omega,\lambda)$ to $A^p(\Omega, \lambda)$. 

If $f\in C_c(\Omega) $, then ${\bm{Q}}f = \wt{\bm{Q}}f \in A^p(\Omega, \lambda)$ and by Theorem~\ref{T:monomials-form-a-Schauder-basis} the Laurent series expansion of $\wt{\bm{Q}}f$ given by \eqref{eq-qflaurent} converges (as a sequence of square partial sums) in $A^p(\Omega,\lambda)$:
\begin{equation}\label{eq-qrearrange}
\wt{\bm{Q}}f = \lim_{N\to \infty} \sum_{\substack{\abs{\alpha}_\infty\leq N\\\alpha\in \cs_p(\Omega,\lambda) }}\wt{a}_\alpha(f) e_\alpha. 
\end{equation}
For a general $g \in L^p(\Omega, \lambda)$, $\wt{\bm{Q}}g \in A^p(\Omega,\lambda)$ and so again by Theorem~\ref{T:monomials-form-a-Schauder-basis},
\begin{equation}\label{eq-wtbmqg} 
\wt{\bm{Q}}g = \lim_{N\to \infty} \sum_{\substack{\abs{\alpha}_\infty\leq N\\\alpha\in \cs_p(\Omega,\lambda)}} {a}_\alpha(\wt{\bm{Q}}g) e_\alpha.
\end{equation} 
It follows that on $C_c(\Omega)$ we have the identity $a_\alpha\circ{\bm{Q}}= \wt{a}_\alpha$. 
This relation extends by continuity to give  $a_\alpha\circ{\wt{\bm{Q}}}= \wt{a}_\alpha$ as functionals on $L^p(\Omega, \lambda)$. 
Then \eqref{eq-wtbmqg} becomes
\[
\wt{\bm{Q}}g = \lim_{N\to \infty} \sum_{\substack{\abs{\alpha}_\infty\leq N\\\alpha\in \cs_p(\Omega,\lambda) }}\wt{a}_\alpha(g) e_\alpha.
\]
In other words, $\wt{\bm{Q}}$ is the MBP, as we wanted to show. 
\end{proof}

\begin{proof}[Proof of Theorem~\ref{thm-mbpmbk}]
Since the MBP exists, by Proposition~\ref{prop:MBP-is-a-basis-projection} the operator $\bm{Q}$ of \eqref{eq-qdef} satisfies $L^p$-estimates. 
Then, by the continuity of point-evaluation in $A^p(\Omega,\lambda)$, for each $z \in \Omega$ the map $g \mapsto\bm{Q}g(z)$ is a bounded linear functional on $L^p(\Omega, \lambda)$. 
Formula \eqref{eq-qdef} representing this functional now shows that $K_{p,\lambda}^\Omega(z,\cdot)\in L^q(\Omega, \lambda)$.
Standard techniques of real analysis (cutting off and mollification) gives us a sequence $\{f_j\}\subset C_c(\Omega)$ such that $f_j\to f$ in $L^p(\Omega, \lambda)$. 
Therefore for each $z\in \Omega$, the sequence $\{K^\Omega_{p,\lambda}(z,\cdot)f_j(\cdot) \}\subset C_c(\Omega)$ converges in $L^1(\Omega,\lambda)$ to the limit $K^\Omega_{p,\lambda}(z,\cdot)f(\cdot)$. 
Since integration against the weight $\lambda$ is a bounded linear functional on $L^1(\Omega,\lambda)$, we obtain \eqref{E:def-of-MBP} in the limit.
\end{proof}

\section{The one dimensional case}\label{sec-onedim}
In this section we compute Monomial Basis Kernels on the unit disc $\D$ and punctured unit disc $\D^*$ -- specifically, the MBKs of the spaces $A^p(\D,\mug)$ and $A^p(\D^*,\mug)$ where $\mug(z) = |z|^\gamma$. 
From these formulas it is shown that the corresponding Monomial Basis Projections are absolutely bounded integral operators.  
We begin with a more general computation of certain subkernels that are needed in Section~\ref{S:MonomialPolyhedra}. 

\subsection{Arithmetic progression subkernels on \texorpdfstring{$\D$}{D} and \texorpdfstring{$\D^*$}{Dstar}}\label{SS:MBK-formulas}
Let $a, b \in \Z$ with $b$ positive, $U = \D$ or $\D^*$, $1<p<\infty$ and $\mug(z) = \abs{z}^\gamma$, $\gamma \in \R$.  
Consider the set of integers
\begin{equation}\label{E:arithmetic-progression_A_{a,b}}
\ca(U,p,\gamma,a,b) = \{\alpha\in \Z: \alpha\equiv a \mod b\} \cap \cs_p(U, \mug),
\end{equation}
where as usual, $\cs_p(U,\mug)\subset \Z$ is the set of $\alpha$ such that $e_\alpha \in A^p(U,\mug)$. 
Notice that $a$ is determined only modulo $b$, so we can always assume that $0\leq a \leq b-1$.
Notice also that if $b=1$ and $a=0$ we have $\ca(U,p,\gamma,0,1) =  \cs_p(U, \mug)$.
We now identify $\ca(U,p,\gamma,a,b)$ with an arithmetic progression:

\begin{proposition}\label{P:SetA-arith-prog}
Let $U,p,\gamma,a,b$ be as above. There is an integer $\theta$
such that 
\begin{equation}\label{E:SetA-alt-expression}
\ca(U,p,\gamma,a,b) = \{\theta+\nu b: \nu\geq0, \nu\in \Z\}.
\end{equation}
\end{proposition}
\begin{proof}
Let $U = \D^*$.
We claim that $\alpha \in \cs_p(\D^*,\mug)$ if and only if $p\alpha + \gamma + 2 > 0$. 
Indeed,
\begin{equation}\label{E:monomial-norm-on-disc}
\norm{e_\alpha}_{p,\mug}^p = \int_{\D^*} |z|^{p\alpha+\gamma}\,dV = 2\pi\int_0^1 r^{p\alpha+\gamma+1}\,dr = \frac{2\pi}{p\alpha + \gamma + 2},
\end{equation}
as long as $p\alpha + \gamma + 2 > 0$, otherwise the integral diverges.
Now let $\theta$ be the smallest integer such that 
(i) $\theta \equiv a \mod b$, and (ii) $p\theta+\gamma+2>0$.  Clearly 
\eqref{E:SetA-alt-expression} holds.

The case $U = \D$ is nearly identical, but the condition that $e_\alpha$
belongs to $A^p(\D, \mug)$ means that $\alpha$ must be nonnegative.
If $\theta$ is the smallest integer in the set $\cs_p(\D,\mug)$, it is determined now by three conditions: (i) $\theta \equiv a \mod b$, (ii) $p\theta+\gamma+2>0$, and (iii) $\theta \ge 0$.
\end{proof}
\begin{remark}
For $U,p,\gamma,a,b$ as above (with $0\le a \le b-1$), we can determine $\theta$ explicitly:
\begin{equation*}
\theta = 
\begin{cases}
a + b\ell, & U = \D^* \\
\max\{a + b\ell,a\}, & U = \D,
\end{cases}
\quad \text{where} \quad
\ell = \left\lfloor -\frac{\gamma+2}{pb} - \frac{a}{b} + 1 \right\rfloor.
\end{equation*}
\hfill $\lozenge$
\end{remark}

Now define for $z,w \in U$ the {\em arithmetric progression subkernel}
\begin{equation}\label{D:arith-prog-kern}
k^{U}_{p,\gamma,a,b}(z,w)
= \sum_{\alpha \in \ca(U,p,\gamma,a,b)} \frac{e_\alpha(z) \ol{\chi_p^*e_\alpha(w)}} {\norm{e_\alpha}_{p,\mug}^p} 
= \sum_{\alpha \in \ca(U,p,\gamma,a,b)} \frac{t^\alpha}{\norm{e_\alpha}_{p,\mug}^p},
\end{equation}
where $\chi_p^*$ is defined by \eqref{eq-chipstar} and $t = z \ol{w} \abs{w}^{p-2}$.   Notice that
$k^U_{p,\gamma, 0,1}$ is the MBK of $A^p(U,\mug)$.

\begin{proposition}\label{P:arith-prog-kernel}
For $z,w\in U$ and other notation as specified above, we have 
\begin{equation}\label{E:arith-prog-kern-form}
k^{U}_{p,\gamma,a,b}(z,w) = \frac{t^{\theta}}{2\pi} \cdot \frac{(p\theta + \gamma +2) - (\gamma+2+p(\theta-b))t^{b}}{(1-t^b)^2}.
\end{equation}
\end{proposition}

\begin{proof} 
The calculation in \eqref{E:monomial-norm-on-disc} shows that if $\alpha \in \cs_p(U,\mug)$, then
\begin{equation*}
\norm{e_{\alpha}}^p_{p,\mug} =  \frac{2\pi}{p\alpha + \gamma + 2}.
\end{equation*}
Now combining \eqref{D:arith-prog-kern} with Proposition \ref{P:SetA-arith-prog}, we see that
\begin{align*}
k^{U}_{p,\gamma,a,b}(z,w) 
=  \sum_{\alpha \in \ca(U,p,\gamma,a,b)} \frac{t^\alpha}{\norm{e_\alpha}_{p,\mug}^p} 
&= \frac{t^\theta}{2\pi} \sum_{\nu=0}^\infty (p(\theta + b\nu) + \gamma + 2)t^{b\nu} \\
&= \frac{t^\theta}{2\pi} \left(p\sum_{\nu=0}^\infty (b\nu+1) t^{b\nu} + (p \theta + \gamma + 2 -p)\sum_{\nu=0}^\infty t^{b\nu} \right) \notag.
\end{align*}
Writing this in closed form yields \eqref{E:arith-prog-kern-form}.
    
\end{proof}
\begin{corollary}\label{C:kernel-bound}
The arithmetic progression kernel $k^U_{p,\gamma,a,b}$ admits the bound
\begin{equation*}
\left|k^{U}_{p,\gamma,a,b}(z,w)\right| \le C \frac{(|z||w|^{p-1})^{\theta}}{|1-z^b\ol{w}^b|w|^{(p-2)b}|^2},
\end{equation*}
where $C>0$ is independent of $z,w \in U$.
\end{corollary}
\begin{proof}
This follows from \eqref{E:arith-prog-kern-form}, on noting that $(p\theta + \gamma +2)$ is necessarily positive.
\end{proof}

Setting $a=0, \,b=1$ in Proposition~\ref{P:arith-prog-kernel} yields the  MBKs of $A^p(\D^*,\mug)$ and $A^p(\D,\mug)$:

\begin{corollary}\label{C:MBK-calculation-(punctured)disc}
Let $\gamma \in \R$, $\mug(z)=\abs{z}^\gamma$ and $t= z\ol{w}\abs{w}^{p-2}$.  The Monomial Basis Kernels of $A^p(\D^*,\mug)$ and $A^p(\D,\mug)$ are given, respectively, by 
\begin{enumerate}[wide]
\item $\displaystyle{K^{\D^*}_{p,\mug}(z,w)=  \frac{1}{2\pi} \cdot \frac{(p\ell+\gamma+2)t^\ell- (\gamma+2 +p(\ell-1))t^{\ell+1} }{(1-t)^2}}$, where $\ell = \floor*{-\frac{\gamma+2}{p}+1}$. \label{E:MBK-formula-punctured-disc1}
\item $\displaystyle{K^{\D}_{p,\mug}(z,w)=  \frac{1}{2\pi} \cdot \frac{(pL+\gamma+2)t^L- (\gamma+2 +p(L-1))t^{L+1} }{(1-t)^2}}$, where $L  = \max\{\ell,0\}$. \label{E:MBK-formula-punctured-disc2}
\end{enumerate}
\end{corollary}

\subsection{Two tools}
We now recall two important results.

\begin{proposition}\label{P:Schurs-test}
For $1 \leq j \leq N$, let $D_j$ be a domain in $\rl^{n_j}$, let $K_j:D_j\times D_j \to [0,\infty)$ be a positive kernel on $D_j$, and let $\lambda^j$ be an a.e. positive weight on $D_j$. 
Suppose that for each $j$, there exist a.e. positive measurable functions $\phi_j, \psi_j$ on $D_j$ and constants $C_1^{j}, C_2^{j} > 0$ such that the following two estimates hold:
\begin{enumerate}[wide]
\item For every $z \in D_j$, $\displaystyle{\int_{D_j} K_j(z,w) \psi_j(w)^q \lambda^j(w)\, dV(w) \le C_1^j \phi_j(z)^q}$.
\item For every $w \in D_j$, $\displaystyle{\int_{D_j} \phi_j(z)^pK_j(z,w)  \lambda^j(z)\, dV(z) \le C_2^j \psi_j(w)^p}$.
\end{enumerate}

Now let $D = D_1 \times \cdots \times D_N$ be the product of the domains, let $K(z,w) = \prod_{j=1}^N K_j(z_j,w_j)$, where $z_j, w_j\in D_j$, $z=(z_1,\dots, z_N)\in D$, $w=(w_1, \dots, w_N)\in D$, and  let $\lambda(w) = \prod_{j=1}^N \lambda^j(w_j)$.
Then the following operator is bounded on $L^p(D,\lambda)$:
\[
\bm{T}f(z) = \int_D K(z,w) f(w) \lambda(w) dV(w).
\]
\end{proposition}
\begin{proof}
When $N=1$, this is the classical Schur's test for boundedness of integral operators on $L^p$-spaces (see \cite[Theorem~3.6]{zhubook}). 
The case $N \geq 2$ reduces to the case $N=1$, if we let $\phi(z) = \prod_{j=1}^N \phi_j(z_j)$ and $\psi(z) = \prod_{j=1}^N \psi_j(z_j)$ and use the Tonelli-Fubini theorem to represent integrals over $D$ as repeated integrals over the product representations.
\end{proof}

\begin{proposition}[Lemma 3.4 of \cite{EdhMcN16}; also see \cite{ForRud74} for  $\beta = 0$]\label{P:standard-Bergman-disc-estimate}
Let $U=\D$ or $\D^*$,  $0<\epsilon <1$ and $\beta> -2$. There exists $C>0$ such that
\begin{equation}
\int_U \frac{(1-|w|^2)^{-\epsilon}}{|1-z\ol{w}|^2}|w|^{\beta}\,dV(w) \le C (1-|z|^2)^{-\epsilon}.
\end{equation}
\end{proposition}

\subsection{\texorpdfstring{$L^p$}{lp}-boundedness of operators}\label{SS:Lp-bounded-operators}

We now prove that arithmetic progression subkernels represent absolutely bounded operators. 
In particular, the existence and absolute boundedness of the Monomial Basis Projections of $A^p(\D^*,\mug)$ and $A^p(\D,\mug)$ are established.

\begin{proposition}\label{P:abs-boundedness-arith-prog-operators}
Define the following auxiliary functions on $U$:
\begin{align*}
\phi(z) = |z|^{\frac{\theta}{q}}(1-|z|^{2b})^{-\frac{1}{pq}}, \qquad  
\psi(w) = |w|^{\frac{\theta}{q}}(1-|w|^{2b(p-1)})^{-\frac{1}{pq}}.
\end{align*}
There exist constants $C_1, C_2>0$, such that the following estimates hold:
\begin{enumerate}
\item For $z\in U$, $\displaystyle{\int_{U} \abs{k^{U}_{p,\gamma,a,b}(z,w)}  \psi(w)^q \mug(w)\, dV(w) \le C_1 \phi(z)^q}$. \label{E:arith-kern-est1} \\
\item For $w\in U$, $\displaystyle{\int_{U} \phi(z)^p\abs{k^{U}_{p,\gamma,a,b}(z,w)}  \mug(z)\, dV(z) \le C_2 \psi(w)^p}$. \label{E:arith-kern-est2}
\end{enumerate}
\end{proposition}

\begin{proof}
Throughout this proof, $C$ will denote a positive number depending on $p,\gamma,a,b$ but independent of $z,w\in U$. 
Its value will change from step to step. 

From the kernel bound in Corollary~\ref{C:kernel-bound}, we obtain
\begin{align}
\int_{U} \big|k_{p,\gamma,a,b}^{U}(z,w)\big| \psi(w)^q \mug(w)\, dV(w) 
 &\le C \int_{U} \frac{(|z||w|^{p-1})^{\theta}}{|1-z^b\ol{w}^b|w|^{(p-2)b}|^2} \psi(w)^q \mug(w)\, dV(w) \notag \\
&= C|z|^{\theta} \int_{U} \frac{ \left(1-|w|^{2b(p-1)}\right)^{-\frac{1}{p}} }{|1-z^b\ol{w}^b|w|^{(p-2)b}|^2} |w|^{p\theta + \gamma} \, dV(w). \label{E:integral-aux-1}
\end{align}
Set $\zeta = w^b|w|^{(p-2)b}$, so $|\zeta| = |w|^{(p-1)b}$,  $|w| = |\zeta|^{\frac{q-1}{b}}$ and $dV(w) = \left(\frac{q-1}{b^2}\right) |\zeta|^{\frac{2(q-1)}{b}-2}dV(\zeta)$. 
This change of variable shows
\begin{align}\label{E:integral-aux-2}
\eqref{E:integral-aux-1} \le C |z|^{\theta} \int_{U} \frac{(1-|\zeta|^2)^{-\frac{1}{p}}}{|1-z^b\ol{\zeta}|^2} |\zeta|^{\frac{q\theta}{b} + \frac{(\gamma+2)(q-1)}{b} - 2}\, dV(\zeta).
\end{align}
This integral converges if and only if $q\theta + (\gamma+2)(q-1) > 0$.
Multiplying by the positive number $\frac{p}{q}$, we see this condition is equivalent to requiring that $p\theta + \gamma + 2 > 0$, which is guaranteed to hold.
Indeed, in the proof of Proposition~\ref{P:SetA-arith-prog}, $\theta$ was shown to be the smallest integer such that (i) $\theta \equiv a \mod b$, and (ii) $p\theta + \gamma + 2 > 0$.
 Now apply Proposition~\ref{P:standard-Bergman-disc-estimate}:
\begin{align*}
\eqref{E:integral-aux-2} \le C |z|^{\theta} (1-|z|^{2b})^{-\frac{1}{p}} = C \left( |z|^{\frac{\theta}{q}}(1-|z|^{2b})^{-\frac{1}{pq}} \right)^q = C \phi(z)^q,
\end{align*}
giving us estimate \eqref{E:arith-kern-est1} upon taking the final constant $C$ to be $C_1$. 
Now consider
\begin{align}
\int_{U} \big|k_{p,\gamma,a,b}^{U}(z,w)\big| \phi(z)^p \mug(z)\, dV(z)\le C \int_{U} \frac{(|z||w|^{p-1})^{\theta}}{|1-z^b\ol{w}^b|w|^{(p-2)b}|^2} \phi(z)^p \mug(z)\, dV(z) \notag \\
= C |w|^{(p-1)\theta}\int_{U} \frac{(1-|z|^{2b})^{-\frac{1}{q}}}{|1- w^b|w|^{(p-2)b}\ol{z}^b|^2}|z|^{(1 + \frac{p}{q})\theta+\gamma}\,dV(z). \label{E:integral-aux-3}
\end{align}
Set $\xi = z^b$, which says that $|z| = |\xi|^{\frac{1}{b}}$ and $dV(z) = b^{-2}|\xi|^{\frac{2}{b}-2}dV(\xi)$. 
This shows that
\begin{align}\label{E:integral-aux-4}
\eqref{E:integral-aux-3} \le C |w|^{(p-1)\theta}\int_{U} \frac{(1-|\xi|^{2})^{-\frac{1}{q}}}{|1- w^b|w|^{(p-2)b}\ol{\xi}|^2}\,|\xi|^{\frac{p\theta}{b}+\frac{\gamma+2}{b}-2}\,dV(\xi),
\end{align}
This integral converges since $p\theta + \gamma + 2 > 0$ (this is the same condition as before). 
Now apply Proposition~\ref{P:standard-Bergman-disc-estimate} again to see
\begin{equation*}\eqref{E:integral-aux-4} 
\le C |w|^{(p-1)\theta}\big(1-|w|^{2b(p-1)}\big)^{-\frac{1}{q}} 
= C\left( |w|^{\frac{\theta}{q}}(1-|w|^{2b(p-1)})^{-\frac{1}{pq}} \right)^p
= C \psi(z)^p,
\end{equation*}
giving estimate \eqref{E:arith-kern-est2} upon taking the final constant $C$ to be $C_2$.
\end{proof}
 
\begin{corollary}\label{C:abs-subprojections-on-U}
The following operator is bounded on $L^p(U,\mug)$:
\begin{equation}\label{E:absolute-mbp-def-1-dim}
\bm{T}^{U}_{p,\gamma,a,b}(f)(z) = \int_U \left|k^{U}_{p,\gamma,a,b}(z,w) \right|f(w)\mug(w)dV(w).
\end{equation}
\end{corollary}
\begin{proof}
Estimates \eqref{E:arith-kern-est1} and \eqref{E:arith-kern-est2} in Proposition \ref{P:abs-boundedness-arith-prog-operators} allow for immediate application of Proposition~\ref{P:Schurs-test} with $N=1$, proving the result.
\end{proof}

\begin{corollary}\label{C:AMBO-MBP-existence-on-U}
The Monomial Basis Projections of the spaces $A^p(\D,\mug)$ and $A^p(\D^*,\mug)$ exist and are absolutely bounded.
\end{corollary}
\begin{proof}
Absolute boundedness (which by Theorem \ref{thm-mbpmbk} implies existence) follows from Corollary~\ref{C:abs-subprojections-on-U} on noting that the MBK of $A^p(U,\mug)$ coincides with the subkernel $k^{U}_{p,\gamma,0,1}$.
\end{proof}


\section{Transformation formula}\label{sec-monomialtransformation}

\subsection{The canonical-bundle pullback} \label{sec-sharppullback}
If $\phi:\Omega_1\to \Omega_2$ is a  finite-sheeted holomorphic map of domains in $\cx^n$, and $f$ is a function on $\Omega_2$, we define a function on $\Omega_1$ by setting 
\begin{equation}\label{eq-sharp}
\phi^\sharp(f)= f\circ \phi\cdot \det \phi',
\end{equation}
where $\phi'(z):\cx^n\to \cx^n$ is the complex derivative of the map $\phi$ at $z\in \Omega_1$.
If we think of $\Omega_1, \Omega_2$ as subsets of $\rl^{2n}$ and $\phi$ as a smooth mapping, we can also consider the $2n\times 2n$ real Jacobian $D\phi$ of $\phi$.
Using the well-known relation $\det D\phi =\abs{\det\phi'}^2$ between the two  Jacobians, we see that $\phi^\sharp$ is a continuous linear mapping of Hilbert spaces $\phi^\sharp: L^2(\Omega_2)\to L^2(\Omega_1)$, and restricts to a map  $A^2(\Omega_2) \to A^2(\Omega_1)$. 
We will refer to $\phi^\sharp$ as the \emph{canonical-bundle pullback} induced by $\phi$, or informally as the \emph{$\sharp$-pullback}, in order to distinguish it from a second pullback to be introduced in Section \ref{SS:Density-bundle-pullbacks}.  
If $\phi$ is a biholomorphism, then $\phi^\sharp$ is an isometric isomorphism of Hilbert spaces $L^2(\Omega_2)\cong L^2(\Omega_1)$ that restricts to an isometric isomorphism $A^2(\Omega_2) \cong A^2(\Omega_1)$. 

\subsection{Proper maps of quotient type}\label{SSS:Prop-mps-quotient-type}
In the classical theory of holomorphic mappings, one considers proper holomorphic mappings, and extends the biholomorphic invariance of Bergman spaces to such mappings via Bell's transformation formula (see \cite{bellduke,belltransactions,difo,bellcat}). 
In our applications, we are concerned with a specific class of proper holomorphic mappings. We begin with the following definition (see \cite{bcem}):
\begin{definition}\label{def-quotienttype}
Let $\Omega_1, \Omega_2\subset\cx^n$ be domains, let $\Phi: \Omega_1\to \Omega_2$ be a proper holomorphic mapping and $\Gamma\subset \mathrm{Aut}(\Omega_1)$ a finite group of biholomorphic automorphisms of $\Omega_1$. 
We say $\Phi$ is of \emph{quotient type with respect to} $\Gamma$ if
\begin{enumerate}
    \item there exist closed lower-dimensional complex-analytic subvarieties $Z_j\subset \Omega_j,\, j=1,2$, such that $\Phi$ restricts to a covering map $\Phi: \Omega_1\setminus Z_1 \to \Omega_2\setminus Z_2$, and
    \item for each $z\in \Omega_2\setminus Z_2$, the action of $\Gamma$ on $\Omega_1$ restricts to a transitive action on the fiber $\Phi^{-1}(z)$.
\end{enumerate}
The group $\Gamma$ is called {\em the group of deck transformations of $\Phi$.}
\end{definition}

The restricted map $\Phi:\Omega_1\setminus Z_1 \to \Omega_2 \setminus Z_2$ is a \emph{regular} covering map (see \cite[page~135 ff.]{massey}); i.e., it gives rise to a biholomorphism between $\Omega_2\setminus Z_2$ and the quotient $(\Omega_1\setminus Z_1)/\Gamma$, where it can be shown  that $\Gamma$ acts properly and discontinuously on $ \Omega_1\setminus Z_1$. 
It follows that $\Gamma$ is the full group of deck transformations of the  covering map $\Phi:\Omega_1\setminus Z_1 \to \Omega_2\setminus Z_2$, and that this covering map has exactly $\abs{\Gamma}$ sheets, where $|\Gamma|$ is the size of the group $\Gamma$. 
By analytic continuation, the relation $\Phi\circ\sigma=\Phi$ holds for each $\sigma$ in $\Gamma$ on all of $\Omega_1$.

\begin{definition}\label{eq-gammainv-sharp}
Given a domain $\Omega\subset \cx^n$, a group $\Gamma \subset \mathrm{Aut}(\Omega)$ and a space $\mathfrak{F}$ of functions on $\Omega$, we define
\begin{equation}\label{eq-gammainv}
   \left[\mathfrak{F}\right]^\Gamma = \{f\in \mathfrak{F} : f=  \sigma^\sharp(f) \text{ for all } \sigma \in \Gamma \}, 
\end{equation}
where $\sigma^\sharp$ is the canonical-bundle pullback induced by $\sigma$ as in \eqref{eq-sharp}. 
We say that functions in this space are said to be {\em $\Gamma$-invariant in the $\sharp$ sense}, or simply {$\sharp$-invariant}. 
\end{definition}

If $L, M$ are Banach spaces, by a {\em homothetic isomorphism} ${\bm{T}}:L\to M$ we mean a bijection such that there is a  $C>0$ satisfying
\begin{equation}\label{D:homothetic-isom}
\norm{{\bm{T}} f}_M = C \norm{f}_L, \qquad \text{ for every } f \in L.
\end{equation}

Fix $1<p<\infty$ and consider a proper holomorphic mapping $\Phi:\Omega_1 \to \Omega_2$ of quotient type with respect to group $\Gamma$. 
Define the function
\begin{equation}\label{D:weight-lambda_p}
    \lambda_p  = |\det \Phi'|^{2-p}.
\end{equation}
This function arises as a weight in naturally occuring $L^p$-spaces. 
Indeed, in Proposition 4.5 of \cite{bcem} it was shown that the map 
\begin{equation}\label{E:Phi-sharp-homothetic-isomorphism}
    \Phi^\sharp:L^p(\Omega_2) \to [L^p(\Omega_1, \lambda_p)]^\Gamma
\end{equation}
is a homothetic isomorphism with 
\begin{equation}\label{E:homothetic-scaling}
\norm{\Phi^\sharp(f)}^p_{L^p(\Omega_1,\lambda_p)} = \abs{\Gamma}\cdot\norm{f}^p_{L^p(\Omega_2)},
\end{equation}
which restricts to a homothetic isomorphism of the holomorphic subspaces
\begin{equation}\label{E:Phi-sharp-holomorphic-homothetic-isomorphism}
    \Phi^\sharp:A^p(\Omega_2) \to [A^p(\Omega_1, \lambda_p)]^\Gamma.
\end{equation}

\subsection{Density-bundle pullbacks}\label{SS:Density-bundle-pullbacks}
Let $\Omega_1, \Omega_2$ be open sets in $\R^d$, and $\phi:\Omega_1\to \Omega_2$ a smooth map. 
Given a function $f$ on $\Omega_2$, define the {\em density-bundle pullback}, or {\em $\flat$-pullback}, of $f$ to be the function on $\Omega_1$ given by
\begin{equation}\label{eq-flatdef}
\phi_\flat f = f\circ \phi \cdot \abs{\det D\phi}^\frac{1}{2},
\end{equation}
where as before, $D\phi$ denotes the $d\times d$ Jacobian matrix of $\phi$.
From the change of variables formula, it follows that if $\phi:\Omega_1\to \Omega_2$ is a diffeomorphism, then the induced map $\phi_\flat:L^2(\Omega_2)\to L^2(\Omega_1)$ is an isometric isomorphism of Hilbert spaces. 
When $\Omega_1, \Omega_2$ are domains in a complex Euclidean space $\cx^n$ and the map $\phi:\Omega_1\to \Omega_2$ is holomorphic, then 
\begin{equation}\label{eq-flat}
\phi_\flat f = f\circ \phi \cdot \abs{\det \phi'},
\end{equation}
where as before, $\phi'$ denotes the complex derivative. 

\begin{definition}
Given a domain $\Omega\subset \cx^n$, group $\Gamma \subset \mathrm{Aut}(\Omega)$ and function space $\mathfrak{F}$ consisting of functions on $\Omega$, define the subspace
\begin{equation}\label{eq-gammainv2}
\left[\mathfrak{F}\right]_\Gamma = \{f\in \mathfrak{F} : f=  \sigma_\flat(f) \text{ for all } \sigma \in \Gamma \},
\end{equation}
where $\sigma_\flat$ is the density-bundle pullback in \eqref{eq-flat}.  
Functions in $\left[\mathfrak{F}\right]_\Gamma$ are said to be {\em $\Gamma$-invariant in the $\flat$ sense}, or simply {\em $\flat$-invariant} when $\Gamma$ is clear from context.
\end{definition}

The behavior of the $\flat$-pullback regarding $L^p$-spaces and $\flat$-invariant functions is analogous to the $\sharp$-pullback regarding $L^p$-spaces and $\sharp$-invariant functions:
 \begin{proposition}\label{P:Phi-flat-homot-isom}
Let $1<p<\infty$, $\Omega_1,\Omega_2$ be domains in $\C^n$ and $\Phi:\Omega_1 \to \Omega_2$ be a proper holomorphic map of quotient type with respect to the group $\Gamma \subset \mathrm{Aut}(\Omega_1)$.  Then
 \begin{equation}
\Phi_\flat:L^p(\Omega_2) \to \left[L^p(\Omega_1,\lambda_p)\right]_\Gamma
\end{equation}
is a homothetic isomorphism.
\end{proposition}
\begin{proof}
Let $f \in L^p(\Omega_2)$.  By Definition \ref{def-quotienttype}, there exist varieties $Z_1 \subset\Omega_1$, $Z_2\subset\Omega_2$ such that $\Phi: \Omega_1\backslash Z_1 \to \Omega_2\backslash Z_2$ is a regular covering map of order $|\Gamma|$.  
Using the change of variables formula (accounting for the fact that $\Phi$ is a $\abs{\Gamma}$-to-one mapping), we see
\begin{equation}\label{E:Phi_flat-norm-equivalence}
    |\Gamma| \norm{f}^p_{L^p(\Omega_2)} = |\Gamma| \int_{\Omega_2\backslash Z_2} |f|^p\, dV = \int_{\Omega_1\backslash Z_1} |f \circ \Phi|^p |\det \Phi'|^2 \, dV = \norm{\Phi_\flat(f)}^p_{L^p(\Omega_1, \lambda_p)},
\end{equation}
which shows $\Phi_\flat(f) \in L^p(\Omega_1,\lambda_p)$.  Observe also that for any $\sigma \in \Gamma$,
\begin{align*}
\sigma_\flat(f \circ \Phi \cdot |\det \Phi'|)  
= f \circ (\Phi \circ \sigma) \cdot |\det (\Phi \circ \sigma)'|
= f \circ \Phi \cdot |\det \Phi'|,
\end{align*}
showing that $\Phi_\flat(f) \in \left[L^p(\Omega_1,\lambda_p)\right]_\Gamma$.  This shows $\Phi_\flat$ is a homothetic isomorphism of $L^p(\Omega_2)$ onto a subspace of $\left[ L^p(\Omega_1 ,\lambda_p) \right]_\Gamma$.  

It remains to show that this image is the full space. 
By a partition of unity argument, it is sufficient to show that a function $g\in \left[L^p(\Omega_1,\lambda_p)\right]_\Gamma$ is in the range of $\Phi_\flat$, provided the support of $g$ is contained in a set of the form $\Phi^{-1}(U)$, where $U$ is an connected open subset of $\Omega_2\setminus Z_2$ evenly covered by the covering map $\Phi$.
Notice that $\Phi^{-1}(U)$ is a disjoint collection of connected open components each biholomorphic to $U$, and if $U_0$ is one of them, $\Phi^{-1}(U)$ is the disjoint union $\bigcup_{\sigma\in \Gamma} \sigma(U_0)$. 
Let $\Psi:U\to U_0$ be the local inverse of $\Phi$ onto $U_0$.
Define $f_0$ on $U$ by $f_0 = \Psi_\flat\left( g|_{U_0}\right)$.
We claim that $f_0$ is defined independently of the choice of the component $U_0$ of $\Phi^{-1}(U)$. Indeed, any other choice is of the form $\sigma(U_0)$ for some $\sigma\in \Gamma$  and the
corresponding local inverse is $\sigma\circ \Psi$. But we have
\[
(\sigma\circ \Psi)_\flat \left( g|_{\sigma(U_0)}\right) = \Psi_\flat \circ \sigma_\flat \left( g|_{\sigma(U_0)}\right)
= \Psi_\flat \left( g|_{U_0}\right)  =f_0,
\]
where we have used the fact that $ \sigma_\flat g=g$ since $g\in \left[ L^p(\Omega_1,\lambda_p)\right ]_\Gamma$. A partition of unity argument completes the proof. 
\end{proof}

\subsection{Monomial maps}\label{SS:MonomialMaps}

Consider an $n\times n$ integer matrix $A$ whose element in the $j$-th row and $k$-th column of $A$ is $a_k^j$. 
Let $a^j$ denote the $j$-th row of $A$, and $a_k$ the $k$-th column. 
Letting the rows of $A$ correspond to monomials $e_{a^j}$, define for $z\in \C^n$ the {\em matrix power}
\begin{equation}\label{eq-matrixpower}
z^A = \begin{pmatrix}
e_{a^1}(z)\\\vdots\\ e_{a^n}(z)
\end{pmatrix}
= \begin{pmatrix}
z_1^{a^1_1}z_2^{a^1_2}\cdots z_n^{a^1_n}\\\vdots\\ z_1^{a^n_1}z_2^{a^n_2}\cdots z_n^{a^n_n}
\end{pmatrix},
\end{equation}
provided each component is defined. 
Define the {\em monomial map} $\Phi_A$ to be the rational map on $\C^n$ given by
\begin{equation}\label{D:monomial-map}
    \Phi_A(z) = z^A.
\end{equation} 

The following properties of monomial maps are known in the literature and references to their proofs are given at the end of the list.
Three pieces of notation must first be explained:
The element-wise exponential map $\exp:\cx^n \to (\cx^*)^n$ is given by $\exp(z) = (e^{z_1},\dots, e^{z_n})$;
if $z = (z_1,\dots,z_n),\,\, w = (w_1,\dots,w_n)$ are points in $\C^n$, define their component-wise product to be $z \odot w = (z_1 w_1, z_2 w_2, \dots, z_n w_n)$; 
$\one \in \Z^{1 \times n}$ is a row vector with $1$ in each component.
\begin{subequations}
\begin{enumerate}
\item The following formula generalizes the familiar power-rule:
\begin{equation}\label{E:det-of-(Ph:i_A)'} 
    \det\Phi'_A = \det A\cdot{e_{\one A-\one}}.
\end{equation}
\item If $A$ is an invertible $n \times n$ matrix of nonnegative integers, then $\Phi_{A}: \C^n \to \C^n$ is a proper holomorphic map of quotient type with respect to the group
\begin{equation}\label{eq-sigmanu}
    \Gamma_A = \{ \sigma_\nu: \sigma_\nu(z)=  \exp\left(2\pi i A^{-1}\nu\right)\odot z, \,\, \nu \in \Z^{n\times 1} \}.
    \end{equation}
\item The group $\Gamma_A$ has exactly $|\det A|$ elements.
\item The canonical-bundle pullback of the monomial $e_\alpha$ via the element $\sigma_\nu \in \Gamma_A$ is
    \begin{equation}\label{E:phi-sharp-applied-to-monomial-formula}
    \sigma_\nu^\sharp(e_\alpha) = e^{2\pi i (\alpha + \one)A^{-1}\nu} \cdot e_\alpha.
    \end{equation}
\item The set of monomials that are $\Gamma_A$-invariant in the $\sharp$ sense as defined by \eqref{eq-gammainv} is
\begin{equation}\label{eq-invariantmonomials}
\{e_\alpha : \alpha = \beta A - \one, \,\, \beta \in \Z^{1\times n} \}.
\end{equation}
\end{enumerate}
\end{subequations}

\begin{proof}
Property (1) is proved in both \cite[Lemma~4.2]{NagPraDuke09} and \cite[Lemma 3.8]{bcem}.
Properties (2) and (3) can be found in \cite[Theorem 3.12]{bcem}. 
See also \cite{zwonekhab,NagPra21} for related results. 
Properties (4) and (5) are found in \cite[Proposition 6.12]{bcem}.
\end{proof}

 
\subsection{Conditions for the transformation formula}\label{SS:gen-hypothesis}

For the remainder of Section~\ref{sec-monomialtransformation}, we assume the following conditions in the statements of our results:

{\em The domain $\Omega_2 \subset \C^n$ is pseudoconvex and Reinhardt, $A$ is an $n \times n$  matrix of nonnegative integers such that $\det A \not=0$, and $\Omega_1 = \Phi_A^{-1}(\Omega_2)$, the inverse image of $\Omega_2$ under the monomial map $\Phi_A:\C^n\to \C^n$ defined in \eqref{D:monomial-map}.}

This set-up has several immediate consequences:
\begin{enumerate}
\item We obtain by restriction a proper holomorphic map
\begin{equation*}
    \Phi_A:\Omega_1 \to \Omega_2,
\end{equation*}
which is of quotient type with respect to the group $\Gamma_A$ defined in \eqref{eq-sigmanu}. 
\item The domain $\Omega_1$ is pseudoconvex and Reinhardt.
\item The weight $\lambda_p$ from \eqref{D:weight-lambda_p} is given by
\begin{equation}\label{E:lambda_p-monomial-map}
    \lambda_p(\zeta) = |\det \Phi_A'(\zeta)|^{2-p} = |\det A|^{2-p} \prod_{k=1}^n |\zeta_k|^{(\one \cdot a_k - 1)(2-p)},
\end{equation}
where as before $\one \in \Z^{1 \times n}$ has $1$ in each component and $a_k$ is the $k$-th column of $A$. 
\item By Proposition~\ref{prop-admissible}, the weight $\lambda_p$ 
is admissible in the sense of Section~\ref{sec-notation}.
\item By \eqref{E:Phi-sharp-homothetic-isomorphism} the canonical-bundle pullback gives a homothetic isomorphism 
\[
\Phi_A^{\sharp}:L^p(\Omega_2) \to [L^p(\Omega_1,\lambda_p)]^{\Gamma_A},
\]
which by \eqref{E:Phi-sharp-holomorphic-homothetic-isomorphism} restricts to a homothetic isomorphism of the holomorphic subspaces
\[
\Phi_A^\sharp : A^p(\Omega_2) \to [A^p(\Omega_1,\lambda_p)]^{\Gamma_A}.
\]
\end{enumerate}

\subsection{\texorpdfstring{$\Gamma$}{Gamma}-invariant subkernel}\label{SS:Gamma-invMBPandMBK}
Assuming the conditions and set-up established in Section~\ref{SS:gen-hypothesis}, define the following subset of $p$-allowable indices which are $\Gamma$-invariant in the $\sharp$ sense. (We often suppress reference to the matrix $A$ in our notation, writing $\Phi_A = \Phi$, $\Gamma_A = \Gamma$, etc.)
\begin{equation}\label{E:gamma-inv-allowable-indices}
\cs_p^\Gamma(\Omega_1,\lambda_p) = \{\alpha \in \cs_p(\Omega_1,\lambda_p): \sigma^\sharp(e_\alpha) = e_\alpha \textrm{ for all } \sigma \in \Gamma  \}.
\end{equation}
We use this to define the ``$\Gamma$-invariant subkernel" of the Monomial Basis Kernel:
\begin{equation}\label{E:Gamma-inv-MBK}
K_{p,\lambda_p,\Gamma}^{\Omega_1}(z,w) 
 = \sum_{\alpha \in \cs_{p}^{\Gamma}(\Omega_1,\lambda_p)} \frac{e_\alpha(z) \ol{\chi_p^*e_\alpha(w)}}{\norm{e_\alpha}_{p,\lambda_p}^p}.
\end{equation}

\begin{proposition}\label{P:Gamma-inv-allowable-indices}
The following sets are equal
\[
\left\{ e_\beta : \beta \in \cs_p^\Gamma(\Omega_1, \lambda_p) \right\} = \big\{ \tfrac{1}{\det A }\Phi^\sharp(e_\alpha): \alpha\in  \cs_p(\Omega_2) \big\}.
\]
\end{proposition}
\begin{proof}
Thinking of $\alpha$ as an element of $\Z^{1 \times n}$, a computation shows that $e_\alpha \circ \Phi_A = e_{\alpha A}$.  Thus
$\Phi^\sharp(e_\alpha) = (\det A) e_{(\alpha + \one)A-1}$, so we have
\begin{equation}\label{E:Schauder-simp}
\big\{ \tfrac{1}{\det A}\Phi^\sharp(e_\alpha): \alpha \in \cs_p(\Omega_2) \big\}
=\{ e_{(\alpha+\one)A-\one}:\alpha\in \cs_p(\Omega_2)\}.
\end{equation}
Since the image of $A^p(\Omega_2)$ under $\Phi^\sharp$ is the space $[A^p(\Omega_1,\lambda_p)]^\Gamma$, we see
\[
\{e_{(\alpha+\one)A-\one}:\alpha\in \cs_p(\Omega_2)\}
\subset \{ e_\beta: \beta\in \cs_p(\Omega_1,\lambda_p), \,\, \sigma^\sharp(e_\beta) = e_\beta \text{ for all } \sigma\in\Gamma\}.
\]
But since the map $\Phi^\sharp:A^p(\Omega_2) \to [A^p(\Omega_1,\lambda_p)]^\Gamma$ is linear, $\Phi^\sharp(f)$ must have more than one term in its Laurent expansion if $f$ has more than one term in its Laurent expansion.  Thus
\begin{align*}
\{ e_{(\alpha+\one)A-\one}:\alpha\in \cs_p(\Omega_2)\}
&= \{ e_\beta: \beta\in \cs_p(\Omega_1,\lambda_p), \,\, \sigma^\sharp(e_\beta)=e_\beta \text{ for all } \sigma\in\Gamma\}\nonumber\\
&= \left\{ e_\beta : \beta \in \cs_p^\Gamma(\Omega_1, \lambda_p) \right\},
\end{align*}
completing the proof.
\end{proof}

\subsection{Transforming operators with positive kernels}\label{SS:AMBO-trans-law}

We prove here a transformation law for the ``absolute" operator involving the MBK:
\begin{equation}\label{E:unweighted-AMBO}
(\bm{P}^{\Omega_2}_{p,1})^+ f(z) = \int_{\Omega_2}\abs{K^{\Omega_2}_{p,1}(z,w)}  f(w) \,dV(w), \qquad f\in C_c(\Omega_2).  
\end{equation}
This operator is defined on $C_c(\Omega_2)$, but can be extended to $L^p(\Omega_2)$ when $L^p$-estimates are shown to hold.
Define a related operator using the $\Gamma$-invariant subkernel from \eqref{E:Gamma-inv-MBK}:
\begin{equation}\label{E:invariant-AMBO}
    (\bm{P}^{\Omega_1}_{p,\lambda_p,\Gamma})^+ f(z) = \int_{\Omega_1}  \abs{K^{\Omega_1}_{p,\lambda_p,\Gamma}(z,w)}f(w) \lambda_p(w)dV(w), \qquad f \in C_c(\Omega_1).
\end{equation}
These operators are closely related via the $\flat$-pullback of Section \ref{SS:Density-bundle-pullbacks}:

\begin{theorem}\label{T:Abs-boundedness-equivalence}
The following statements are equivalent:
\begin{enumerate}
\item $(\bm{P}^{\Omega_2}_{p,1})^+$ extends to a bounded operator $(\bm{P}^{\Omega_2}_{p,1})^+: L^p(\Omega_2) \to L^p(\Omega_2)$.
\item $(\bm{P}^{\Omega_1}_{p,\lambda_p,\Gamma})^+$ extends to a bounded operator
$(\bm{P}^{\Omega_1}_{p,\lambda_p,\Gamma})^+: [L^p(\Omega_1,\lambda_p)]_\Gamma \to [L^p(\Omega_1,\lambda_p)]_\Gamma$.
\end{enumerate}
When these equivalent statements hold, 
\begin{equation}\label{E:AbsMBO-and-Phi_flat}
    \Phi_\flat \circ (\bm{P}_{p,1}^{\Omega_2})^+ = (\bm{P}_{p,\lambda_p,\Gamma}^{\Omega_1})^+ \circ \Phi_\flat
\end{equation}
as operators on $L^p(\Omega_2)$, which is to say, that the following diagram commutes
\begin{equation}\label{Diagram-absoulte-MBP}
\begin{tikzcd}			
{L}^p(\Omega_2) \arrow[r,"{\Phi}_\flat"] \arrow[r,swap,"\cong"] \arrow[d, "(\bm{P}_{p,1}^{\Omega_2})^+"] 
& L^p(\Omega_1, \lambda_p)_\Gamma \arrow[d,"(\bm{P}^{\Omega_1}_{p,\lambda_p,\Gamma})^+"] 
\\
L^p(\Omega_2)
\arrow[r,"{\Phi}_{\flat}"]
& L^p(\Omega_1, \lambda_p)_\Gamma.
\end{tikzcd}
\end{equation}
\end{theorem}
The following kernel transformation formula can be thought of as a generalization of the classical biholomorphic transformation formula for the Bergman kernel.
\begin{proposition}\label{P:MBK-proper-trans-law}
The Monomial Basis Kernel admits the transformation law
\begin{equation}\label{E:MBK-proper-trans-law}
    K^{\Omega_1}_{p,\lambda_p,\Gamma}(z,w) = \frac{1}{|\Gamma|} \det \Phi'(z) \cdot K^{\Omega_2}_{p,1}(\Phi(z),\Phi(w)) \cdot \frac{|\det \Phi'(w)|^p}{\det \Phi'(w)}.
\end{equation}
\end{proposition}
\begin{proof}
Starting from the series representation for $K^{\Omega_2}_{p,1}(z,w)$ in \eqref{E:MBK-alt-form}, we have
\begin{align}
    K_{p,1}^{\Omega_2}(\Phi(z),\Phi(w)) &= \sum_{\alpha\in\cs_p(\Omega_2)} \frac{e_\alpha(\Phi(z))\ol{e_\alpha(\Phi(w))} | e_\alpha(\Phi(w))|^{p-2} }{\norm{e_\alpha}^p_{L^p(\Omega_2)}} \notag \\
    &= |\Gamma| \sum_{\alpha\in\cs_p(\Omega_2)} \frac{e_\alpha(\Phi(z))\ol{e_\alpha(\Phi(w))} | e_\alpha(\Phi(w))|^{p-2} }{\norm{\Phi^\sharp(e_\alpha)}^p_{L^p(\Omega_1,\lambda_p)}}, \label{E:MBK-comp1}
\end{align}
since by \eqref{E:homothetic-scaling}, the homothetic isomorphism $\Phi^\sharp$ scales norms uniformly for each $f \in L^p(\Omega_2)$ by
$|\Gamma|\cdot\norm{f}_{L^p(\Omega_2)}^p = \norm{\Phi^\sharp(f)}^p_{L^p(\Omega_1,\lambda_p)}$.  Now use the definition of $\Phi^\sharp$ to write
\begin{align}
    \eqref{E:MBK-comp1} &= |\Gamma| \frac{\det \Phi'(w)}{\det \Phi'(z)|\det \Phi'(w)|^p} \sum_{\alpha\in\cs_p(\Omega_2)} \frac{\Phi^\sharp (e_\alpha)(z)\ol{\Phi^\sharp(e_\alpha)(w)} | \Phi^\sharp(e_\alpha)(w)|^{p-2} }{\norm{\Phi^\sharp(e_\alpha)}^p_{L^p(\Omega_1,\lambda_p)}} \notag \\ 
    &= |\Gamma| \frac{\det \Phi'(w)}{\det \Phi'(z)|\det \Phi'(w)|^p} \sum_{\beta\in\cs_p^\Gamma(\Omega_1,\lambda_p)} \frac{e_\beta(z)\ol{e_\beta(w)}|e_\beta(w)|^{p-2}}{\norm{e_\beta}^p_{L^p(\Omega_1,\lambda_p)}} \label{E:MBK-comp2} \\
    &= |\Gamma| \frac{\det \Phi'(w)}{\det \Phi'(z)|\det \Phi'(w)|^p} \cdot K^{\Omega_1}_{p,\lambda_p,\Gamma}(z,w). \label{E:MBK-comp3}
\end{align}
Equation \eqref{E:MBK-comp2} follows from Proposition \ref{P:Gamma-inv-allowable-indices}, and \eqref{E:MBK-comp3} follows from the definition of the $\Gamma$-invariant MBK given in \eqref{E:Gamma-inv-MBK}.
This completes the proof.
\end{proof}

\begin{proof}[Proof of Theorem~\ref{T:Abs-boundedness-equivalence}]
Proposition \ref{P:Phi-flat-homot-isom} and \eqref{E:Phi_flat-norm-equivalence} show that $\Phi_\flat: L^p(\Omega_2) \to L^p(\Omega_1,\lambda_p)_\Gamma$ is a homothetic isomorphism with $\norm{\Phi_\flat f}^p_{L^p(\Omega_1,\lambda_p)} = |\Gamma| \norm{f}^p_{L^p(\Omega_2)}$. Now for $f \in C_c(\Omega_2)$,
\begin{align}
    \Phi_\flat \circ (\bm{P}^{\Omega_2}_{p,1})^+ f (z) &= |\det \Phi'(z)| \int_{\Omega_2}  \abs{K^{\Omega_2}_{p,1}(\Phi(z),w)} f(w)\,dV(w) \notag\\
    &= \frac{|\det \Phi'(z)|}{|\Gamma| } \int_{\Omega_1} \abs{K^{\Omega_2}_{p,1}(\Phi(z),\Phi(w))}f(\Phi(w))  \cdot |\det \Phi'(w)|^2 \,dV(w) \notag \\
    &= \int_{\Omega_1} \abs{K^{\Omega_1}_{p,\lambda_p,\Gamma}(z,w)}\Phi_\flat f(w) \lambda_p(w)\,dV(w) \label{E:AbsMBO-kernel-trans-formula}\\
    &= (\bm{P}^{\Omega_1}_{p,\lambda_p,\Gamma})^+ \circ \Phi_\flat f (z) \notag.
\end{align}
Equality in \eqref{E:AbsMBO-kernel-trans-formula} uses the kernel transformation law \eqref{E:MBK-proper-trans-law}, and the final line makes sense since the properness of $\Phi$ guarantees $\Phi_\flat f \in \left[C_c(\Omega_1)\right]_\Gamma$.
The fact that $C_c(\Omega_2)$ is dense in $ L^p(\Omega_2)$, along with the fact that its image $\Phi_\flat\left(C_c(\Omega_2)\right) = \left[C_c(\Omega_1)\right]_\Gamma$ is dense in $\left[L^p(\Omega_1,\lambda_p) \right]_\Gamma$ shows that statements $(1)$ and $(2)$ are equivalent.  
When these statements hold, equation \eqref{E:AbsMBO-and-Phi_flat} and Diagram \eqref{Diagram-absoulte-MBP} follow immediately.
\end{proof}


\section{Monomial Polyhedra}\label{S:MonomialPolyhedra}

In this section we prove Theorem \ref{T:MBP-AbsBoundedness}, which says that if $\Uu$ is a monomial polyhedron and $1<p<\infty$, the Monomial Basis Projection of $A^p(\Uu)$ is absolutely bounded.
As discussed in Section~\ref{SS:Intro-monomial-polyhedra}, this stands in contrast with the limited $L^p$-regularity of the Bergman projection.

\subsection{Matrix representation}\label{SS:matrix-representation}

We denote the spaces of  row and column vectors with integer entries by $\Z^{1 \times n}$ and $\Z^{n \times 1}$, respectively. 
Suppose $B = (b_k^j ) \in M_{n \times n}(\Z)$ is a matrix of integers with $\det B\not=0$, with rows written as $b^j = (b^j_1,\dots, b^j_n) \in \Z^{1 \times n}$. 
Define 
\begin{equation}\label{eq-udef}
\mathscr{U}_B = \left\{z\in\cx^n:  |e_{b^j}(z)|  < 1, \quad 1\leq j \leq n  \right\},
\end{equation}
and call it the  monomial polyhedron associated to the matrix $B$, provided it is bounded.
This gives a compact notation for the domains  defined in Section~\ref{SS:Intro-monomial-polyhedra}

The matrix $B$ in \eqref{eq-udef} is far from unique. 
If $B'$ is obtained from $B$ by permuting rows or by multiplying any row by a positive integer, then $\Uu_B=\Uu_{B'}$. 
We recall the following observation, originally proved in \cite[Proposition~3.2]{bcem}:
\begin{proposition}\label{P:ConditionsOnMatrixB}
Suppose that $\Uu_B$ is a bounded monomial polyhedron as in \eqref{eq-udef}, where $\det B \neq 0 $.  
Without loss of generality we may assume
\begin{enumerate}
\item $\det{B} > 0$. \label{E:detB-positive}
\item each entry in the inverse matrix $B^{-1}$ is nonnegative. \label{E:labels-Binv-nonneg}
\end{enumerate}
\end{proposition}

Given the monomial polyhedron $\Uu_B$, we will assume for the rest of the paper that $B$ satisfies both properties \eqref{E:detB-positive} and \eqref{E:labels-Binv-nonneg} of Proposition \ref{P:ConditionsOnMatrixB}.
Observe that Cramer's rule combined with property \eqref{E:labels-Binv-nonneg} says that the adjugate $A = (\det B) B^{-1}$ is a matrix of nonnegative integers.

The following representation of monomial polyhedra as quotients was first proved in \cite[Theorem~3.12]{bcem}.

\begin{proposition}\label{P:covering-mono-polyh}
Let  $A = (\det B) B^{-1}\in M_{n \times n}(\Z)$. 
There exists a product domain 
\begin{equation}\label{eq-Omegaproduct}
\Omega = U_1 \times \dots \times U_n \subset \C^n,
\end{equation}
each factor $U_j$ either a unit disc $\D$ or a unit punctured disc $\D^*$, such that the monomial map $\Phi_{A}: \C^n \to \C^n$ of \eqref{D:monomial-map} restricts to a proper holomorphic  map $\Phi_A:\Omega \to \Uu_B$.  
This map is of quotient type with respect to group $\Gamma_A$, which is given in \eqref{eq-sigmanu}.
\end{proposition}

The conditions of Section~\ref{SS:gen-hypothesis} are satisfied, if we take  $\Omega_1=\Omega$, $\Omega_2=\Uu_B$, and $A, \Phi_A, \Gamma_A$ as above in Proposition~\ref{P:covering-mono-polyh}.  
In the present situation, the source domain $\Omega_1=\Omega$ is a product and the weight $\lambda_p=\abs{\det\Phi_A'}^{2-p}$ of \eqref{E:lambda_p-monomial-map} admits a tensor product structure:
\begin{equation}\label{eq-lambdap-monopoly}
\lambda_p(\zeta) = \abs{\det \Phi_A'(\zeta)}^{2-p} = (\det A)^{2-p} \prod_{j=1}^n \mugj(\zeta_j),
\end{equation}
where $\mugj$ is the weight on $U_j$ given by
\begin{equation}\label{E:component-weights}
\mugj(z) = |z|^{\gamma_j}, \qquad \mathrm{where} \qquad \gamma_j = (\one \cdot a_j -1)(2-p),
\end{equation}
$\one \in \Z^{1 \times n}$ is the row vector with $1$ in each component and $a_j \in \Z^{n \times 1}$ the $j$-th column of $A$.
We can remove the absolute value from $\det A$ since $\det A = (\det B)^n\cdot \frac{1}{\det B} = \det B^{n-1}>0$.

\subsection{Absolute boundedness of the Monomial Basis Projection}\label{SS:Lp-bddness-AMBO}
We now give a decomposition of the the $\Gamma$-invariant subkernel defined in \eqref{E:Gamma-inv-MBK}. 

\begin{proposition}\label{P:Gamma-inv-MBK-decomp}
Let $d = \det A$ (a positive integer). 
The $\Gamma$-invariant subkernel defined in \eqref{E:Gamma-inv-MBK} admits the decomposition
\begin{equation}\label{E:Gamma-inv-MBK-decomp1}
K^\Omega_{p,\lambda_p,\Gamma}(z,w) = \sum_{i=1}^{d^{n-1}} K_i(z,w),
\end{equation}
where each $K_i$ is a tensor product of $n$ arithmetic progression subkernels defined in \eqref{D:arith-prog-kern}:
\begin{equation}\label{E:Gamma-inv-MBK-decomp2}
K_i(z,w) = d^{p-2} \prod_{j=1}^n k^{U_j}_{p,\gamma_j,\alpha_{i,j},d}(z_j,w_j),
\end{equation}
where the $\gamma_j$ is determined by \eqref{E:component-weights} and $\alpha_{i,j}\in \Z/d\Z$ is determined by the group $\Gamma$.
\end{proposition}
\begin{proof}
Following \eqref{E:Gamma-inv-MBK}, the $\Gamma$-invariant subkernel $K_{p,\lambda_p,\Gamma}^{\Omega}(z,w)$ is found by summing over the $p$-allowable indices, $\Gamma$-invariant in the $\sharp$ sense.  From \eqref{E:gamma-inv-allowable-indices}, this set can be written as
\begin{align}
\cs_p^\Gamma(\Omega,\lambda_p) &= \{\alpha \in \cs_p(\Omega,\lambda_p): \sigma^\sharp(e_\alpha) = e_\alpha \textrm{ for all } \sigma \in \Gamma  \} = \cs_p(\Omega,\lambda_p)\cap [\Z^n]^\Gamma \label{eq-zngamma},
\end{align}
where  $[\Z^n]^\Gamma$ is defined to be the subset of $\Z^{1\times n}$ consisting of exactly those indices for which the corresponding monomials are $\Gamma$-invariant, i.e.,
\[
[\Z^n]^\Gamma=\{\alpha\in \Z^{1\times n}: \sigma^\sharp(e_\alpha) = e_\alpha \textrm{ for all } \sigma \in \Gamma\}. 
\]
By \eqref{eq-invariantmonomials}, we see that
$[\Z^n]^\Gamma = \{\alpha\in \Z^{1\times n}:\alpha = \beta A -\one,\, \beta \in \Z^{1\times n}\}$, so after translating by $\one$, we have 
\[
[\Z^n]^\Gamma+\one=\Z^{1\times n}A= \{\beta A: \beta\in \Z^{1\times n}\} \subset \Z^{1\times n}.
\]
We make two observations: 
first, it is known (see Lemma 3.3 of \cite{NagPra21}) that $\Z^{1\times n}A$ is a sublattice of $\Z^{1 \times n}$ with index 
\begin{equation*}
    \abs{ \Z^{1 \times n}/(\Z^{1\times n}A)} = \det A = d.
\end{equation*}
Second, we claim that $\Z^{1\times n}A$ contains $d\,\Z^{1 \times n} = \{d\beta: \beta\in \Z^{1 \times n}\}$ as a sublattice.
Consider a vector $v = d  y$, for some $y \in \Z^{1 \times n}$ and check that $v \in \Z^{1\times n}A$. 
Since $A$ is invertible, there is a solution $x \in \Q^{1 \times n}$ with $v = d y = x A$.  
Write $A$ in terms of its rows $a^1, \cdots, a^n\in \Z^{1\times n}$ as $A=[a^1, \cdots, a^n]^T.$
Cramer's rule shows the $j$-th component of $x$ is
\begin{align*}
x_j = \frac{\det\left([a^1,\cdots,a^{j-1},d y, a^{j+1},\cdots,a^n]^T\right)}{\det A} = \det\left([a^1,\cdots,a^{j-1},y, a^{j+1},\cdots,a^n]^T\right) \in \Z,
\end{align*}
confirming that $x \in \Z^{1 \times n}$, and therefore that $d \, \Z^{1 \times n}$ is a sublattice of $\Z^{1 \times n}A$.  

Since the index $\left|\Z^{1 \times n} / d \, \Z^{1 \times n} \right|=d^n$, the Third Isomorphism Theorem for groups says
\begin{equation*}
\left|\Z^{1 \times n}A / d \, \Z^{1 \times n} \right| = \frac{\left| \Z^{1 \times n} / d\,\Z^{1 \times n} \right|}{\left| \Z^{1 \times n}/ \Z^{1 \times n} A \right|} = d^{n-1}.
\end{equation*}
It now follows that we have a representation of the group $\Z^{1 \times n} A$ as a disjoint union of $d^{n-1}$ cosets of the subgroup $d\,\Z^{1 \times n}$, i.e., there are  $\ell^i\in  \Z^{1 \times n} A $, such that we have
\[ \Z^{1 \times n} A=[\Z^n]^\Gamma+\one= \bigsqcup_{i=1}^{d^{n-1}}(d\,\Z^{1 \times n}+\ell^i), \]
where  $\bigsqcup$ denotes disjoint union. Therefore, we have
\[    
[\Z^n]^\Gamma = \left(\bigsqcup_{i=1}^{d^{n-1}}(d\,\Z^{1 \times n}+\ell^i)\right)-\one
= \bigsqcup_{i=1}^{d^{n-1}}\left(d\,\Z^{1 \times n}+(\ell^i-\one)\right).
\]
Fix an $i, 1\leq i \leq d^{n-1}$ and write $\ell^i=(\ell^i_1,\dots, \ell^i_n)$ with $\ell^i_j\in \Z$. 
Then we have
\begin{align}
d\,\Z^{1 \times n}+(\ell^i-\one)&=\{(d\cdot \nu_1+\ell^i_1 -1, \dots, d\cdot \nu_n+\ell^i_n -1): \nu_1, \dots, \nu_n\in \Z\}\notag\\
&= \prod_{j=1}^n\{\alpha\in \Z: \alpha\equiv \ell^i_j-1 \mod d\}, \label{eq-prod1}
\end{align}
where in the last line we have the Cartesian product of $n$ sets of integers. 

We now analyze the other intersecting set $\cs_p(\Omega,\lambda_p)$ in \eqref{eq-zngamma}.
Let $\alpha\in \Z^n$. 
Combining the representation of $\lambda_p$ from \eqref{eq-lambdap-monopoly} with the fact that $e_\alpha(z)=\prod_{j=1}^n e_{\alpha_j}(z_j)$, we can write the norm of $e_\alpha$ on $\Omega$ in terms of the norms of the $e_{\alpha_j}$ on the factors $U_j$:
 \begin{equation}\label{eq-tonelli}
 \norm{e_\alpha}_{L^p(\Omega,\lambda_p)}^p
 = d^{2-p} \prod_{j=1}^n \norm{e_{\alpha_j}}_{L^p(U_j\mugj)}^p.
 \end{equation}
The left-hand side is finite, i.e., $\alpha\in \cs_p(\Omega, \lambda_p)$, if and only if each factor on the right-hand side is finite, i.e., for each $1\leq j \leq n$ we have $\alpha_j\in \cs_p(U_j, \mugj)$.
Consequently we obtain a Cartesian product representation of the set
\begin{equation}\label{eq-prod2}
\cs_p(\Omega, \lambda_p)=\prod_{j=1}^n \cs_p(U_j, \mugj).
\end{equation}
Therefore by \eqref{eq-zngamma}, we have
\[
\cs_p^\Gamma(\Omega,\lambda_p) 
= \cs_p(\Omega,\lambda_p)\cap 
\left(\bigsqcup_{i=1}^{d^{n-1}}\left((d\,\Z^{1 \times n}+\ell_i)-\one\right)\right)
= \bigsqcup_{i=1}^{d^{n-1}} \mathscr{L}_i,
\]
where 
\begin{align}
\mathscr{L}_i
&=  \cs_p(\Omega,\lambda_p) \cap \left((d\,\Z^{1 \times n}+\ell_i)-\one\right)
& \text{by definition}\notag\\
&= \left(\prod_{j=1}^n \cs_p(U_j, \mugj)\right)\bigcap \left(\prod_{j=1}^n\{\alpha\in \Z: \alpha\equiv \ell^i_j-1 \mod d\}\right)
&\text{ by \eqref{eq-prod1} and \eqref{eq-prod2}}\notag\\
&= \prod_{j=1}^n \left(\cs_p(U_j,\mugj)\cap
\{\alpha\in \Z: \alpha\equiv \ell^i_j-1 \mod d\} \right)\notag\\
&=  \prod_{j=1}^n \ca(U_j,p,\gamma_j,\ell^{i}_j-1,d), \label{eq-Li}
\end{align}
and the last equality follows from the definition \eqref{E:arithmetic-progression_A_{a,b}}.  
We now define
\begin{equation}\label{E:orthog-lattice-subkernel}
K_i(z,w) = \sum_{\alpha \in {\mathscr{L}}_i} \frac{e_\alpha(z) \ol{\chi_p^*e_\alpha(w)}} {\norm{e_\alpha}_{p,\lambda_p}^p},
\end{equation}
which immediately gives \eqref{E:Gamma-inv-MBK-decomp1}, since absolute convergence
permits rearrangement of the series defining $ K_{p,\lambda_p,\Gamma}^{\Omega_1}$.
Now from \eqref{eq-tonelli}, we see that for $\alpha\in \mathscr{L}_i$ we have
\begin{equation}\label{eq-tensorealpha}
\frac{e_\alpha(z) \ol{\chi_p^*e_\alpha(w)}} {\norm{e_\alpha}_{p,\lambda_p}^p} 
= d^{p-2} \prod_{j=1}^n \frac{e_{\alpha_j}(z_j) \ol{\chi_p^*e_{\alpha_j}(w_j)}} {\norm{e_{\alpha_j}}_{p,\mugj}^p},
\end{equation}
where for each $j$, we have $\alpha_j\in \ca(U_j,p,\gamma_j,\ell^{i}_j-1,d)$, and on the right hand side $\chi_p:\cx\to \cx$ is the one-dimensional version of the map \eqref{eq-chip}.
Using \eqref{eq-Li} and \eqref{eq-tensorealpha}, we can rearrange the sum \eqref{E:orthog-lattice-subkernel} as
\begin{align}
K_i(z,w) 
&= d^{p-2} \prod_{j=1}^n \left(\sum_{\alpha_j\in \ca(U_j,p,\gamma_j,\ell^{i}_j-1,d) }\frac{e_{\alpha_j}(z_j) \ol{\chi_p^*e_{\alpha_j}(w_j)}} {\norm{e_{\alpha_j}}_{p,\mugj}^p}\right) \label{eq-rearrange}\\
&= d^{p-2} \prod_{j=1}^n k^{U_j}_{p,\gamma_j,\ell^{i}_j-1,d}(z_j, w_j)\notag
\end{align}
where the rearrangement in \eqref{eq-rearrange} is justified since each of the $n$ factor series on the right hand side is absolutely convergent.
The final line is just the definition given in \eqref{D:arith-prog-kern}.
\end{proof}
\begin{proof}[Proof of Theorem \ref{T:MBP-AbsBoundedness}]
Theorem \ref{T:Abs-boundedness-equivalence} says $(\bm{P}^{\Uu}_{p,1})^+:L^p(\Uu) \to L^p(\Uu)$ is a bounded operator if and only if $(\bm{P}^{\Omega}_{p,\lambda_p,\Gamma})^+:[L^p(\Omega,\lambda_p)]_\Gamma \to [L^p(\Omega,\lambda_p)]_\Gamma$ is bounded.
From \eqref{E:Gamma-inv-MBK-decomp1}, we see that
\begin{equation}\label{E:bound-on-GammaInvMBP}
\abs{K_{p,\lambda_p,\Gamma}^{\Omega}(z,w)} \le \sum_{i=1}^{d^{n-1}} \abs{K_i(z,w)}.
\end{equation}
From formula \eqref{E:invariant-AMBO} defining the operator $(\bm{P}^{\Omega}_{p,\lambda_p,\Gamma})^+$, it would be sufficient to prove that for each $1 \leq i\leq n$, the operator
\[
f \mapsto \int_{\Omega} \abs{K_i(\cdot,w)} f(w) \lambda_p(w)dV(w)
\]
is bounded on (the full space) $L^p(\Omega,\lambda_p)$. 
Formula \eqref{E:Gamma-inv-MBK-decomp2} now gives
\[
\abs{K_i(z,w)} = d^{p-2} \prod_{j=1}^n \abs{k^{U_j}_{p,\gamma_j,\alpha_{i,j},d}(z_j,w_j)}.
\]
Proposition~\ref{P:abs-boundedness-arith-prog-operators} now says that for each $1\leq j \leq n$, there exist functions $\phi_j, \psi_j$ and constants $C_1^j, C_2^j$ such that 
\begin{align*}
&\int_{U_j} \abs{k^{U_j}_{p,\gamma_j,\alpha_{i,j},d}(z,w) } \psi_j(w)^q \mugj(w)\, dV(w) \le C^j_1 \phi_j(z)^q,  \\
&\int_{U_j} \phi_j(z)^p\abs{k^{U_j}_{p,\gamma_j,\alpha_{i,j},d}(z,w)}  \mugj(z)\, dV(z) \le C^j_2 \psi_j(w)^p.
\end{align*}
Proposition~\ref{P:Schurs-test} now finishes the proof.
\end{proof}


\section{Duality theory of Bergman spaces}\label{S:Duality}

\subsection{Properties of the twisting map}\label{SS:Props-of-twisting-map}
In this section, $\Omega$ will denote an arbitrary Reinhardt domain in $\C^n$.
We return now to the twisting map $\chi_p$ introduced in \eqref{eq-chip}, and use it to present a duality theory for Bergman spaces on Reinhardt domains.
This leads to a concrete description for all $1<p<\infty$ of the duals of the $A^p$-Bergman spaces when the Monomial Basis Projection is absolutely bounded; this is new on all monomial polyhedral domains (including the Hartogs triangle), and even new in the case of the punctured disc.

\begin{proposition}\label{P:props-of-twisting-map}
The twisting map $\chi_p:\C^n \to \C^n$ has the following properties.
\begin{enumerate}
\item \label{E:inverse-of-chip-is-chiq} It is a homeomorphism of $\cx^n$ with itself, and its inverse is the map $\chi_q$.
\item \label{item:def-of-eta_q} It is a  diffeomorphism away from the set $\bigcup_{j=1}^n \{z_j=0\}$ and its Jacobian determinant (as a mapping of the real vector space $\cx^n$) is given by
\begin{equation}\label{E:def-of-eta_q}
\eta_p(\zeta) = \det(D\chi_p)= (p-1)^n \abs{\zeta_1\cdot \dots \cdot \zeta_n}^{2p-4}.
\end{equation}
\item \label{eq-chipomega} 
It restricts to a homeomorphism $\chi_p:\Omega\to \Omega^{(p-1)}$ with inverse $\chi_q:\Omega^{(p-1)}\to \Omega$, where $\Omega^{(p-1)}$ is a Reinhardt power of $\Omega$ as in \eqref{E:ReinhardtPower}. 
\end{enumerate}
\end{proposition}
\begin{proof}
For item \eqref{E:inverse-of-chip-is-chiq}, notice that if $w = \chi_p(z)$, then for each $j$ we have
\[ 
w_j\abs{w_j}^{q-2} = z_j \abs{z_j}^{p-2}\cdot \abs{z_j |z_j|^{p-2}}^{q-2} = z_j \abs{z_j}^{p-2+(p-1)(q-2)}= z_j,
\]
since $p-2+(p-1)(q-2)=pq-p-q=0$.
So $\chi_q\circ \chi_p$ is the identity, and similarly $\chi_p\circ \chi_q$ is also the identity.
Item \eqref{item:def-of-eta_q} follows from direct computation.
Item \eqref{eq-chipomega} follows upon noting that in each coordinate, the map $z \mapsto z\abs{z}^{p-2}$ is represented in polar coordinates as $r e^{i\theta}\mapsto r^{p-1} e^{i\theta}$. 
The claim follows from the definition of $\Omega^{(p-1)}$.
\end{proof} 
\begin{proposition}
The Monomial Basis Kernels of $A^p(\Omega)$ and $A^q\big(\Omega^{(p-1)},\eta_q\big)$ are related via the twisting map in the following way:
\begin{equation}\label{eq-twistedconjsym}
K^{\Omega}_{p,1}\left(\chi_q(z),w\right) = \ol{{K^{\Omega^{(p-1)}}_{q,\eta_q}(\chi_p(w),z)}},\qquad z\in \Omega^{(p-1)}, \, w\in \Omega.
\end{equation} 
This ``twisted" symmetry generalizes the conjugate symmetry of the Bergman kernel on $\Omega$.
\end{proposition}
\begin{proof}
Recalling equation \eqref{eq-chistaralpha} above, observe that
\[
\abs{\chi_q^* e_\alpha(\zeta)}^p= \abs{e_\alpha(\chi_q(\zeta))}^p = \abs{e_\alpha(\zeta)}^{(q-1)p} = \abs{e_\alpha(\zeta)}^q.
\]
Now using $\chi_q$ to change of variables, we have
\begin{align*} 
\norm{e_\alpha}^p_{L^p(\Omega)} 
= \int_{\Omega^{(p-1)}} \abs{e_\alpha(\chi_q(\zeta))}^p \eta_q(\zeta) \,dV(\zeta)
= \norm{e_\alpha}^q_{L^q(\Omega^{(p-1)},\eta_q)}, 
\end{align*}
which in particular shows the equality of the sets $\cs_p(\Omega) = \cs_q\big(\Omega^{(p-1)},\eta_q\big)$ of allowable indices. 
Thus, for $z \in \Omega^{(p-1)}$ and $w \in \Omega$, we have
\begin{align*}
K^{\Omega}_{p,1}\left(\chi_q(z),w\right)  
&= \sum_{\alpha \in \cs_p(\Omega)} \frac{e_\alpha(\chi_q(z)) \ol{\chi_p^* e_\alpha(w)}}{\norm{e_\alpha}_{L^p(\Omega)}^p} \\
&= \sum_{\alpha \in \cs_q(\Omega^{(p-1)},\eta_q)} \ol{\frac{ e_\alpha(\chi_p(w)) \ol{\chi_q^* e_\alpha(z)}}{\norm{e_\alpha}^q_{L^q(\Omega^{(p-1)},\eta_q)}} } 
= \ol{{K^{\Omega^{(p-1)}}_{q,\eta_q}(\chi_p(w),z)}}.
\end{align*}
By setting $p=2$, \eqref{eq-twistedconjsym} recaptures the conjugate symmetry of the Bergman kernel.
\end{proof}

\subsection{Adjoints and Duality}\label{SS:adjoints-and-duality}
We now use the map $\chi_p$ to give a ``twisted" $L^2$-style pairing of the spaces $L^p(\Omega)$ and $L^q(\Omega^{(p-1)},\eta_q)$:
\begin{equation}\label{eq-new-pairing}
\{f,g\}_{p} = \int_\Omega f \cdot \ol{\chi_p^*(g)} \,dV, \qquad f \in L^p(\Omega), \quad g \in L^q(\Omega^{(p-1)},\eta_q).
\end{equation}

\begin{proposition}\label{prop-pairing}
The map $(f,g) \mapsto \{f,g\}_{p}$, is an isometric duality pairing of $L^p(\Omega)$ and $L^q \big(\Omega^{(p-1)},\eta_q \big)$.
In other words, through $\{\cdot,\cdot \}_p$ we obtain the dual space identification
\[
L^p(\Omega)' \simeq L^q\big(\Omega^{(p-1)},\eta_q \big),
\]
where the operator norm of the functional $\{\cdot,g\}_p \in L^p(\Omega)'$ is equal to the norm of its representative function $g \in L^q\big(\Omega^{(p-1)},\eta_q \big)$.
\end{proposition}
\begin{proof}
It is a classical fact that the ordinary $L^2$-style pairing of $L^p(\Omega)$ with $L^q(\Omega)$ given by 
\[
(f,h) \mapsto \int_\Omega f \cdot \ol{h}\, dV, \qquad f \in L^p(\Omega), \quad g \in L^q(\Omega)
\]
is an isometric duality pairing. 
Proposition \ref{P:props-of-twisting-map} says that $\chi_q:\Omega^{(p-1)}\to \Omega$ is a diffeomorphism outside a set of measure zero, with inverse $\chi_p:\Omega\to \Omega^{(p-1)}$, itself a diffeomorphism outside a set of measure zero. 
It therefore suffices to show that
\begin{equation}\label{E:chi_p-pullback}
    \chi_q^*: L^q(\Omega) \to L^q(\Omega^{(p-1)},\eta_q)
\end{equation}
is an isometric isomorphism of Banach spaces. 
Calculation shows
\begin{align}
\norm{h}^q_{L^q(\Omega)} 
= \int_{\Omega^{(p-1)}} \abs{h \circ \chi_q(w)}^q \,\eta_q(w) dV(w) 
= \norm{\chi_q^*(h)}^q_{L^q(\Omega^{(p-1)},\eta_q)} \label{E:Chi_p-norm-equiv1}.
\end{align}
Since the inverse map $\chi_p^*$ of $\chi_q^*$ exists, it is surjective and the result follows by the closed-graph theorem.
\end{proof}

\begin{proposition}\label{prop-adjoint}
Suppose the Monomial Basis Projection of $A^p(\Omega)$ is absolutely bounded on $L^p(\Omega)$. 
Then under the pairing $\{\cdot,\cdot\}_p$ defined in \eqref{eq-new-pairing}, its adjoint is the Monomial Basis Projection of $A^q(\Omega^{(p-1)},\eta_q)$, which is itself absolutely bounded in $L^q(\Omega^{(p-1)},\eta_q)$; i.e.,
\[
\left\{\bm{P}^\Omega_{p,1}f,g\right\}_p = \left\{f, \bm{P}^{\Omega^{(p-1)}}_{q,\eta_q}g\right\}_p, \quad \text{for all } \,\,f \in L^p(\Omega), \,\, g \in L^q(\Omega^{(p-1)},\eta_q).
\]
\end{proposition}
\begin{proof}
Suppose that $f\in L^p(\Omega)$ and $g\in L^q(\Omega^{(p-1)}, \eta_q)$:
\begin{align}
\left\{\bm{P}^\Omega_{p,1}f,g\right\}_p
= \int_\Omega \bm{P}^\Omega_{p,1} f \cdot \ol{\chi_p^*g }\,dV 
&= \int_\Omega\left( \int_\Omega K^\Omega_{p,1}(z,w)f(w)\, dV(w) \right) \ol{g(\chi_p(z))}\, dV(z)\label{eq-fubini1} \\
&= \int_\Omega \left( \int_\Omega K^\Omega_{p,1}(z,w)\ol{g(\chi_p(z))} dV(z) \right) f(w)\, dV(w), \label{eq-fubini2}
\end{align}
where the change in order of integration can be justified as follows. 
By the assumption that $\bm{P}^\Omega_{p,1}$ is absolutely bounded on $L^p(\Omega)$, we see that the function on $\Omega$ given by
\[ 
z \longmapsto  \int_\Omega \abs{ K^\Omega_{p,1}(z,w)}\cdot \abs{f(w)}dV(w)  
\]
is in $L^p(\Omega)$.  Since $g\in  L^q(\Omega^{(p-1)}, \eta_q)$, using Tonelli's theorem we see that
\begin{align*}
&\int_{\Omega\times \Omega} \abs{K^\Omega_{p,1}(z,w)g(\chi_p(z)) f(w)} dV(z,w)\\
&=\int_\Omega \left( \int_\Omega\abs{K^\Omega_{p,1}(z,w)}\cdot \abs{f(w)}  dV(z) \right) \abs{g(\chi_p(z))} dV(w) <\infty,
\end{align*}
by Proposition~\ref{prop-pairing}.  
Fubini's theorem gives that \eqref{eq-fubini1} = \eqref{eq-fubini2}.   
Now change variables in the inner integral of \eqref{eq-fubini2} by setting $z = \chi_q(\zeta)$, where $\zeta \in \Omega^{(p-1)}$ to obtain 
\begin{align}
\eqref{eq-fubini2} 
&= \int_{\Omega}\left( \int_{\Omega^{(p-1)}} K^\Omega_{p,1}(\chi_q(\zeta),w)\ol{g(\zeta)}\,\eta_q(\zeta)\, dV(\zeta) \right) f(w) dV(w) \nonumber\\
&= \int_{\Omega} \ol{\left( \int_{\Omega^{(p-1)}} {K^{\Omega^{(p-1)}}_{q,\eta_q}(\chi_p(w), \zeta)} {g(\zeta)}\, \eta_q(\zeta)  dV(\zeta) \right)} f(w) dV(w) \label{eq-just1} \\
&= \int_{\Omega} f(w) \ol{\bm{P}^{\Omega^{(p-1)}}_{q,\eta_q}g(\chi_p(w))}\,dV(w)\label{eq-just2} \\
&=\int_\Omega f \cdot \ol{\chi_p^*\left(\bm{P}^{\Omega^{(p-1)}}_{q,\eta_q}g\right)}dV 
= \left\{f, \bm{P}^{\Omega^{(p-1)}}_{q,\eta_q}g\right\}_p.\nonumber
\end{align}
The second equality above follows from \eqref{eq-twistedconjsym}.  The fact that \eqref{eq-just2} = \eqref{eq-just1} can be justified as follows. 
For $g \in L^q\big(\Omega^{(p-1)}, \eta_q\big)$, the quantity in \eqref{eq-just1} is finite for each $f\in L^p(\Omega)$, since by the above computations it is equal to the finite quantity $\{\bm{P}^\Omega_{p,1}f,g \}_p$. 
Therefore we see that for each $g\in L^q(\Omega^{(p-1)}, \eta_q)$, we have that the function
\[
\left( w \longmapsto \int_{\Omega^{(p-1)}} {K^{\Omega^{(p-1)}}_{q,\eta_q}(\chi_p(w), \zeta)} g(\zeta) \eta_q(\zeta) \,dV(\zeta)\right) \in L^q(\Omega),
\]
so that the linear map
\[ 
g \longmapsto \int_{\Omega^{(p-1)}} {K^{\Omega^{(p-1)}}_{q,\eta_q}(\chi_p(\cdot), \zeta)} g(\zeta) \eta_q(\zeta) \, dV(\zeta)
\]
is bounded from $L^q(\Omega^{(p-1)}, \eta_q)$ to $L^q(\Omega)$ by the closed graph theorem (since the integral operator is {easily} seen to be closed). 
Composing with the (isometric) bounded linear map $\chi_q^*:L^q(\Omega)\to L^q(\Omega^{(p-1)}, \eta_q)$, we see that the operator on $L^q(\Omega^{(p-1)}, \eta_q)$ given by 
\[
g \longmapsto \int_{\Omega^{(p-1)}}  {K^{\Omega^{(p-1)}}_{q,\eta_q}(\cdot, \zeta)} g(\zeta) \eta_q(\zeta)\,  dV(\zeta)
\]
is bounded on $L^q(\Omega^{(p-1)}, \eta_q)$. Now Proposition~\ref{prop:MBP-is-a-basis-projection} shows \eqref{eq-just2} = \eqref{eq-just1}.
\end{proof}


\begin{proposition}\label{prop-dual}
Suppose the Monomial Basis Projection of $A^p(\Omega)$ is absolutely bounded on $L^p(\Omega)$. 
Then the duality pairing of $L^p(\Omega)$ and $L^q(\Omega^{(p-1)}, \eta_q)$ by $\{\cdot,\cdot\}_p$ restricts to a duality pairing of the holomorphic subspaces.
 In other words, we can identify the dual space
\[
A^p(\Omega)' \simeq A^q\big(\Omega^{(p-1)},\eta_q \big).
\]
\end{proposition}
\begin{proof}
We claim that the conjugate-linear continuous map $A^q(\Omega^{(p-1)}, \eta_q )\to A^p(\Omega)'$ given by $h \mapsto \{\cdot, h\}_{p,1}$ is a  homeomorphism of Banach spaces. 
To see surjectivity, let $\phi\in A^p(\Omega)'$, let $\wt{\phi}: L^p(\Omega)\to \cx$ be its Hahn-Banach extension, and let $g\in L^q(\Omega^{(p-1)}, \eta_q)$ be such that $\wt{\phi}(f)=\{f,g\}_{p,1}$. 
The existence of $g$ follows from Proposition~\ref{prop-pairing}. 
We see from Proposition~\ref{prop-adjoint} that for each $f\in A^p(\Omega)$ we have
\[ 
\phi(f)= \wt{\phi}(f)= \{f,g\}_p = \{\bm{P}^{\Omega}_{p,1}f,g\}_p= \{f,\bm{P}^{\Omega^{(p-1)}}_{q,\eta_q}g\}_p
\]
so the surjectivity follows since $\bm{P}^{\Omega^{(p-1)}}_{q,\eta_q}g\in A^q\big(\Omega^{(p-1)}, \eta_q \big)$. 
Now if $h\in A^q\big(\Omega^{(p-1)}, \eta_q \big)$ is in the null-space of this map, i.e., for each $f\in A^p(\Omega)$ we have $\{f,h\}_p=0$, then for $g\in L^p(\Omega)$:
\[
\{g,h\}_p = \{g, \bm{P}^{\Omega^{(p-1)}}_{q,\eta_q}h\}_p=\{\bm{P}^{\Omega}_{p,1}g, h\}_p=0. 
\]
This shows that $h = 0$, so the mapping is injective.
\end{proof}

\subsection{Dual spaces on monomial polyhedra}

The duality pairing in Section \ref{SS:adjoints-and-duality} should be contrasted with the usual Hölder duality pairing of $L^p$ and $L^q$. 
On the disc $\D$, the Hölder pairing restricts to a duality pairing of the holomorphic subspaces, yielding the identification $A^p(\D)' \simeq A^q(\D)$.
On the punctured disc, the Hölder pairing fails to restrict to a holomorphic duality pairing and any attempt to identify $A^p(\D^*)'$ with $A^q(\D^*)$ fails.
This is discussed further in Section~\ref{SS:Berg-proj-holomorphic-dual-spaces}.
For similar results, see \cite{zeljko}.

\begin{theorem}\label{T:dual-space-on-U}
Let $U = \D^*$ or $\D$. The dual space of $A^p(U)$ admits the identification
\[
A^p(U)' \simeq A^q(U,\eta_q),\qquad \eta_q(\zeta) = (q-1)|\zeta|^{2q-4},
\]
via the pairing \eqref{eq-new-pairing}, sending $(f,g) \mapsto \{f,g\}_p$, where $f \in A^p(U)$, $g \in A^q(U,\eta_q)$.
\end{theorem}
\begin{proof}
It was shown in Corollary \ref{C:AMBO-MBP-existence-on-U} that the MBP of $A^p(U)$ is absolutely bounded.
Recalling the definition of a Reinhart power in \eqref{E:ReinhardtPower}, it it clear that in our case $U^{(m)} = U$ for every $m>0$, so in particular for $m = p-1$.  Proposition~\ref{prop-dual} now gives the result.
\end{proof}

The same behavior regarding Reinhardt powers seen on the disc and punctured disc continues to hold on all monomial polyhedra:
\begin{proposition}\label{P:Reinhart-powers-monomial-polys}
Let $\Uu \subset \C^n$ be a monomial polyhedron of the form \eqref{eq-udef}.  Then for each $m>0$, the Reinhardt power $\Uu^{(m)} = \Uu$.
\end{proposition}
\begin{proof}
Write $\Uu = \Uu_B$, where the rows of $B$ are given by $b^j = (b^j_1,\dots,b_n^j) \in \Z^{1 \times n}$.  
From the definition of the Reinhardt power of a domain given in \eqref{E:ReinhardtPower}, we see
\begin{align*}
\Uu^{(m)} &= \{ z \in \C^n : (|z_1|^{\frac{1}{m}},\dots,|z_n|^{\frac{1}{m}}) \in \Uu \} \\
&= \{ z \in \C^n : |e_{b^j}\big(|z_1|^{\frac{1}{m}},\dots,|z_n|^{\frac{1}{m}}\big)| < 1, \,\, 1 \le j \le n \} \\
&= \big\{ z \in \C^n : |e_{b^j}(z) |^{\frac{1}{m}} < 1, \,\, 1 \le j \le n \big\} 
= \big\{ z \in \C^n : |e_{b^j}(z) | < 1, \,\, 1 \le j \le n \big\} = \Uu.
\end{align*}
\end{proof}

\begin{theorem}\label{T:MBP-duality-pairing} 
Let $\Uu$ be a monomial polyhedron in $\C^n$. The dual space of $A^p(\Uu)$ admits the identification
\[
A^p(\Uu)' \simeq A^q(\Uu,\eta_q),\qquad \eta_q(\zeta) = (q-1)|\zeta_1\cdots\zeta_n|^{2q-4},
\]
via the pairing \eqref{eq-new-pairing}, sending $(f,g) \mapsto \{f,g\}_p$, where $f \in A^p(\Uu)$, $g \in A^q(\Uu,\eta_q)$.
\end{theorem}
\begin{proof}
The absolute boundedness of the MBP of $A^p(\Uu)$ seen in Theorem~\ref{T:MBP-AbsBoundedness} allows for the use of Proposition~\ref{prop-dual}. 
In this setting $\Uu^{(p-1)} = \Uu$ by Proposition \ref{P:Reinhart-powers-monomial-polys}, which yields the result.
\end{proof}


\section{Comparing the MBP to the Bergman projection on \texorpdfstring{$L^p$}{Lp}}\label{S:BergmanComparison}

Let $\Omega\subset\cx^n$ be a bounded Reinhardt domain such that the origin lies on its boundary. In even the simplest example, the punctured disc $\D^*=\{z\in \cx: 0<\abs{z}<1\}$, special features of the holomorphic function theory can be seen in the Riemann removable singularity theorem. 
Higher dimensional versions of this phenomenon were noticed by Sibony in \cite{sibony1975} on the Hartogs triangle and later generalized in \cite{sibony}.  

\subsection{The \texorpdfstring{$L^p$}{Lp}-irregularity of the Bergman projection.}\label{SS:Lp-reg-Berg-proj}

In understanding the $L^p$ function theory on $\Omega$, it is instructive to consider the behavior of the sets of $p$-allowable indices introduced in Section~\ref{sec-notation}: $\cs_p(\Omega)=\{\alpha\in \Z^n: e_\alpha\in L^p(\Omega)\}$, 
as $p$ traverses the interval $(1,\infty)$. 
It is clear that the sets can only shrink as $p$ increases, as fewer monomials become integrable due to increase in the exponent $p$ in the integral $\int_\Omega \abs{e_\alpha}^p dV$.
However, the set $\cs_p(\Omega)$ is always nonempty,  since $\N^n \subset \cs_p(\Omega)$, $\Omega$ being bounded. 

For example on the punctured disc, if $p<2$, then $\cs_p(\D^*) = \{\alpha \in \Z:\alpha \geq -1\}$, and if $p\geq 2$, then $\cs_p(\D^*) = \{\alpha \in \Z:\alpha \geq 0\}$. 
The exponent $p=2$ where the set of indices shrinks is a \emph{threshold}.
The $L^p$-irregularity of the Bergman projection is closely related with these thresholds.
It was shown in \cite{bcem}, that on a monomial polyhedron $\Uu$, the Bergman projection is bounded
in $L^p$ if and only if $p \in (q^*,p^*)$, where $p^* = p^*(\Uu)$ is the smallest threshold of $\Uu$ bigger than $2$ and $q^* = q^*(\Uu)$ is its Hölder conjugate. 
Explicit values of $p^*$ and $q^*$ are given in the main theorem of \cite{bcem}; see also Proposition \ref{P:BergmanLpRange-MonoPoly}.

Outside the interval $(q^*,p^*)$, the $L^p$-boundedness of the Bergman projection on the monomial polyhedron $\Uu$ fails in different ways depending on whether $p \geq p^*$ or $p \leq  q^*$.
Since $\Uu$ is bounded, we have $L^p(\Uu)\subset L^2(\Uu)$ if $p \geq p^* > 2$, so the integral operator defining the Bergman projection in \eqref{eq-bergmanproj} is defined for each $f\in L^p(\Uu)$. 
The failure of boundedness of the Bergman projection corresponds to the fact that  there are functions $f\in L^p(\Uu)$ for which the projection $\bm{B}^{\Uu}f$ is not in $A^p(\Uu)$. 
It is easy to give  an explicit example when $\Uu=\h$, the Hartogs triangle.
Suppose $p\ge p^*(\h) = 4$ and let $f(z) = \ol{z}_2$, which is  bounded and therefore in $L^p(\h)$. 
A computation shows that there is a constant $C$ such that  $\bm{B}^\h f(z) = C{z_2}^{-1}\notin L^p(\h)$. 
This idea can be generalized to an arbitrary monomial polyhedron $\Uu$ to show that if $p\geq p^*$, there is a function in $L^p(\Uu)$ which projects to a monomial which is in $L^2(\Uu)$ but not in $L^p(\Uu)$. 
In \cite{chakzeytuncu} the range of the map $\bm{B}^\h:L^p(\h)\to L^2(\h)$ for $p\geq 4$ was identified as a weighted $L^p$-Bergman space strictly larger than $L^p(\h)$, and a similar result holds on any monomial polyhedron. 
Recent work in \cite{HuoWick2020a} shows that $\bm{B}^\h$ is of weak-type (4,4), and this has been extended to generalized Hartogs triangles in \cite{koenig2}. 
For $p \le q^*$, the situation is even worse:

\begin{proposition}\label{P:cant-extend}
Let $1<p \le q^*(\Uu) $ and $z\in \Uu$. 
There is a function $f\in L^p(\Uu)$ such that the integral
\[
\int_{\Uu} B^\Uu(z,w) f(w)\,dV(w)
\]
diverges. 
Consequently, there is no way to extend the Bergman projection to $L^p(\Uu)$ using its integral representation. 
\end{proposition}
\begin{proof}
Let $q$ denote the Hölder conjugate of $p$ so that $q\geq p^*$. 
The holomorphic function on the Reinhardt domain $\Uu$ given by $g(\zeta)= B(\zeta, z)$ has Laurent expansion
\[ 
g(\zeta)= \sum_{\alpha\in \cs_2(\Uu)} \frac{\ol{z}^\alpha}{\norm{e_\alpha}_2}\zeta^\alpha. 
\]
Since $q\geq p^*$, and the set of integrable monomials shrinks at $p^*$, it follows  that there is a monomial $e_\alpha\in A^2(\Uu) \setminus A^q(\Uu)$. 
Since this non-$A^q$ monomial appears in the above Laurent series with a nonzero coefficient, and by Theorem~\ref{T:monomials-form-a-Schauder-basis}, the Laurent expansion of a function in $A^q$ can only have monomials which are in $A^q$, it follows that $g\notin A^q(\Uu)$. 
By symmetry therefore, $B(z,\cdot)\not \in L^q(\Uu)$. 
It now follows that there is a function $f \in L^p(\Uu)$ such that the integral above does not converge.
\end{proof}
When $\Uu=\h$, one can show by explicit computation that if $1<p<\frac{4}{3} = q^*(\h)$, we can take $f(w)=w_2^{-3}$ in the above result for each $z\in \h$. 
It was shown  in \cite{HuoWick2020a} that $\bm{B}^\h$ fails to be weak-type $(\frac{4}{3},\frac{4}{3})$, and this was extended in \cite{koenig2} to generalized Hartogs triangles.
But in light of Proposition~\ref{P:cant-extend}, we see that $\bm{B}^\h$ does not even exist as an everywhere defined operator on $L^{4/3}(\h)$.

In contrast with the above, Theorem~\ref{T:MBP-AbsBoundedness} guarantees that for $1<p<\infty$ and $\Uu$ a monomial polyhedron, that the MBP $\bm{P}_{p,1}^\Uu$ is a bounded operator from $L^p(\Uu)$ onto $A^p(\Uu)$, and Theorem~\ref{thm-mbpmbk} says that for $z \in \Uu$, the function $K^\Uu_{p,1}(z,\cdot) \in L^q(\Uu)$, where $\frac{1}{p}+\frac{1}{q} = 1$.

\subsection{Failure of surjectivity}\label{SS:Inability-to-reproduce}
Even if the Bergman projection can be given a bounded extension to $L^p$, it need not be surjective onto $A^p$ for $p<2$, as one sees in the case of the punctured disc. 
Here, since $A^2(\D^{*})$ and $A^2(\D)$ are identical, the Bergman kernels have the same formula. 
The Bergman projection on $\D^*$ consequently extends to a bounded operator on $L^p(\D^*)$ for every $1<p<\infty$, but fails to be surjective onto $A^p(\D^*)$ for $p \in (1,2)$. 
This happens because the range of the Bergman projection can be naturally identified with $A^p(\D)$, and when $1<p<2$, the space $A^p(\D)$ is a strict subspace of $A^p(\D^*)$ (for example the function $g(z)=z^{-1}$ belongs to $A^p(\D^*) \setminus A^p(\D)$). 
In particular, $\bm{B}^{\D^*}$ is not the identity on $A^p(\D^*)$ and its nullspace is the one-dimensional span of $g(z)=z^{-1}$.
 
On the Hartogs triangle, the Bergman projection is bounded on $L^p(\h)$ for $\frac{4}{3} < p < 4$, but is not surjective onto $A^p(\h)$ for $\frac{4}{3} < p < 2$.
Let $\mathcal{N}\subset A^p(\h)$ be the closed subspace spanned by the monomials in $A^p(\h) \setminus A^2(\h)$. 
One sees from a computation that the monomials in $ A^p(\h) \setminus A^2(\h)$ are $e_\alpha$ with  $\alpha_1\ge0$ and $\alpha_1+\alpha_2 = -2$. 
Then one can verify using orthogonality of $L^p$ and $L^q$ monomials that the nullspace of $\bm{B}^\h$ restricted to $A^p(\h)$ is $\mathcal{N}$.

In contrast, the MBP of $A^p(\Uu)$ accounts for {\em all} monomials appearing in the Banach-space basis $\{e_\beta: \beta \in \cs_p(\Uu) \}$, and Corollary \ref{C:MBP-boundedness/surjectivity} shows that for $1<p<\infty$, $\bm{P}^{\Uu}_{p,1}$ is a bounded {\em surjective} projection of $L^p(\Uu)$ onto $A^p(\Uu)$.

\subsection{The Bergman projection and holomorphic dual spaces}\label{SS:Berg-proj-holomorphic-dual-spaces}

The following is a reformulation of \cite[Theorem~2.15]{ChEdMc19}:

\begin{theorem}\label{E:L2-duality-classic}
Suppose that the following two conditions hold on a domain $U\subset \cx^n$.
\begin{enumerate}
\item The absolute Bergman operator $(\bm{B}^U)^+: L^p(U) \to L^p(U)$ is bounded.
\item The Bergman projection acts as the identity operator on both $A^p(U)$ and $A^q(U)$.
\end{enumerate}
Then the sesquilinear Hölder pairing restricts to a duality pairing of $A^p(U)$ with $A^q(U)$:
\begin{equation}\label{E:L2-pairing}
\langle f,g \rangle = \int_U f \ol{g}\,dV, \qquad f\in A^p(U), \quad g\in A^q(U),
\end{equation}
providing the dual space identification $A^p(U)' \simeq A^q(U)$.
\end{theorem}

Conditions (1) and (2) both hold, for instance, on smoothly bounded strongly pseudoconvex domains (see \cite{PhoSte77} and \cite{catlin}), thus yielding the dual space identification. 
But when one of the conditions (1) or (2) fails, the conclusion can fail. 

On the punctured disc $\D^* \subset \C$, (1) always holds but (2) fails for all $p\not=2$; it can be shown that under the pairing \eqref{E:L2-pairing}, $A^p(\D^*)'$ can only be identified with $A^q(\D^*)$ if $p=q=2$.
On the Hartogs triangle $\h$, (1) holds if $\frac{4}{3}<p<4$, but (2) never holds for a $p$ in this range, as we saw in Section~\ref{SS:Inability-to-reproduce}. The pairing \eqref{E:L2-pairing} is not a duality pairing on $\h$ for $\frac{4}{3}<p<4$ unless $p=2$.
The mapping $A^q(\h)\to A^p(\h)'$
given by the pairing is not injective if $2<p<4$ and not surjective if 
$\frac{4}{3}<p<2$.

In contrast with the above, the duality theory of Section~\ref{SS:adjoints-and-duality} characterizes duals of Bergman spaces of Reinhardt domains via the pairing \eqref{eq-new-pairing} whenever the MBP is absolutely bounded.
We saw that Theorem~\ref{T:dual-space-on-U} gives a concrete description of the dual space of $A^p(\D^*)$, and for monomial polyhedra Theorem \ref{T:MBP-duality-pairing} does the same. 

\bibliographystyle{alpha}
\bibliography{ChakEdh22}
\end{document}